\documentclass[11pt]{article}

\usepackage{amssymb,latexsym}
\usepackage{amsmath,amscd}
\usepackage{theorem}
\usepackage[all]{xy}
\usepackage{url}
\usepackage{stmaryrd}
\usepackage{bm} 
\usepackage{graphics}

\title{\bf On self-similarity of $p$-adic analytic \\ pro-$p$ groups
of small dimension}

\author{
        Francesco Noseda  \thanks{{\tt noseda@im.ufrj.br}.}
        \\[0.1cm]
        Ilir Snopce \thanks{{\tt ilir@im.ufrj.br}. Supported by CAPES
        (grant 88881.145624/2017-01), the Alexander von Humboldt Foundation, FAPERJ and CNPq.}
        \\[0.2cm]
        \footnotesize{Instituto de Matem\'atica - 
        Universidade Federal do Rio de Janeiro}\\
        \footnotesize{Avenida Athos da Silveira Ramos 149 
        (Centro de Tecnologia - Bloco C)}\\
        \footnotesize{21941-909, Rio de Janeiro,  Brazil.}
}

\date{}

\newcommand{\bb}[1]{\mathbb{#1}}
\newcommand{\cl}[1]{\mathcal{#1}}
\newcommand{\mb}[1]{\mathbf{#1}}
\newcommand{\mr}[1]{\mathrm{#1}}

\newcommand{\ra}{\rightarrow}
\newcommand{\rar}{\rightarrow}

\newcommand{\vep}{\varepsilon}

\newcommand{\les}{\leqslant}
\newcommand{\ges}{\geqslant}

\newcommand{\sylow}{\scalebox{0.5}{\mbox{\ensuremath{\displaystyle \triangle}}}}

\newcommand{\ep}{\hfill $\square$} 

\newtheorem{lemma} {Lemma} [section]
\newtheorem{proposition} [lemma] {Proposition}

\newtheorem{theorem} [lemma] {Theorem}
\newtheorem{corollary} [lemma] {Corollary}
\newtheorem{definition}[lemma] {Definition}
\newtheorem{conjecture}[lemma] {Conjecture}

\theorembodyfont{\rm} 

\newtheorem{example}[lemma] {Example}
\newtheorem{remark}[lemma]{Remark}

\newenvironment{proof}{{\sc Proof:}}{
\hfill $\square$}
\newenvironment{proof2}{{\sc Proof:}}{}

\numberwithin{equation}{section}

\theorembodyfont{\it}

\newtheorem{theoremx}{Theorem}
\newtheorem{propositionx}[theoremx]{Proposition}
\newtheorem{conjx}[theoremx]{Conjecture}

\usepackage{array}
\usepackage[colorlinks]{hyperref}

\begin{document} 

\maketitle

\begin{abstract}
Given a torsion-free $p$-adic analytic pro-$p$ group $G$ with $\mr{dim}(G) < p$, 
we show that the self-similar actions of $G$ on regular rooted trees
can be studied through the virtual endomorphisms of the associated  
$\bb{Z}_p$-Lie lattice. 
We explicitly classify 3-dimensional unsolvable $\bb{Z}_p$-Lie 
lattices for $p$ odd, and study their virtual endomorphisms. 
Together with Lazard's correspondence, this
allows us to classify  3-dimensional 
unsolvable torsion-free $p$-adic analytic pro-$p$ groups for $p\ges 5$, 
and to determine which of them admit a 
faithful self-similar action on a $p$-ary tree.  
In particular, we show that no open subgroup of $SL_1^1(\Delta_p)$ admits 
such an action. On the other hand, 
we prove that all the open subgroups of $SL_2^{\sylow}(\bb{Z}_p)$
admit faithful self-similar actions on regular rooted trees. 
\end{abstract}

\let\thefootnote\relax
\let\thefootnote\relax


\section*{Introduction}

The class of groups that admit a faithful self-similar action on some regular 
rooted $d$-ary 
tree $T_d$ 
contains many interesting and important examples such as the Grigorchuk 
2-group \cite{Gri80}, the Gupta-Sidki $p$-groups \cite{GuSi83}, the affine groups 
$\mathbb{Z}^n \rtimes GL_n(\mathbb{Z})$ 
\cite{BrSi98},
groups obtained as iterated monodromy groups of self-coverings 
of the Riemann sphere by post-critically finite rational maps \cite{NekSSgrp},
and so on.
Recently there has been an intensive study on the self-similar actions of other 
important families of groups including abelian groups \cite{BrSi10}, finitely 
generated nilpotent groups \cite{BeSi07},  arithmetic groups \cite{Ka12} and 
groups of type $\mr{FP}_n$ \cite{KoSi}. 
Despite such an active research on the topic of self-similarity in recent years, 
very little is known about the self-similar actions of pro-$p$ groups. 
In this context one may ask the following question.

\medskip

\noindent \textbf{Question.} Which pro-$p$ groups admit a faithful 
self-similar action on a regular rooted $p$-ary  tree?

\smallskip

The pro-$2$ completion of the Grigorchuk group coincides with its 
topological closure as a subgroup of $\textrm{Aut}(T_2)$, so it  
is an example of a pro-$2$ group that admits a self-similar action 
on a binary tree. For similar reasons, the pro-$p$ completions of 
the Gupta-Sidki $p$-groups admit self-similar actions on a $p$-ary tree.   
Self-similar actions of some classes of finite $p$-groups where studied 
in \cite{Su11} and \cite{BaFaFeVa}.  

The main goal of the present paper is to study the self-similar actions of 
torsion-free $p$-adic analytic pro-$p$ groups, and we do so using Lie methods. 
It is not difficult to see that every non-trivial torsion-free $p$-adic analytic
pro-$p$ group of dimension at most 2 
admits a faithful self-similar action on a $p$-ary tree 
(see Proposition \ref{pdim2}), but in dimension 3 or higher the problem becomes
considerably more complex.

We say that a group $G$ is \textbf{self-similar of index} $\bm{d}$ 
if $G$ admits a faithful self-similar action on
$T_d$ that is transitive on the first level; moreover,
we say that $G$ is \textbf{self-similar} if it is self-similar of 
some index $d$.  A \textbf{virtual endomorphism} of $G$
is a group homomorphism $\varphi:D\rar G$ where $D\les G$
is a subgroup of \textit{finite} index; the \textbf{index} of $\varphi$
is defined to be the index of $D$ in $G$ (cf. \cite{NekVirEnd}).
A virtual endomorphism  is said to be \textbf{simple}
if there are no non-trivial normal subgroups of $G$ that are $\varphi$-invariant
(cf. Definition \ref{dinvsubs2}).
It is known that a group $G$ is self-similar of index $d$
if and only if it admits a simple virtual endomorphism of index 
$d$ (see Proposition \ref{pssiffve}). 
We will use the language of  virtual endomorphisms in the context of Lie lattices. Given an $n$-dimensional \textbf{Lie lattice} $L$ over $\bb{Z}_p$, 
a \textbf{virtual endomorphism}  of $L$ is a
Lie algebra homomorphism $\varphi:M\rar L$ where 
$M\subseteq L$ is a subalgebra 
of dimension $n$; the \textbf{index} of $\varphi$ is defined to be the index of $M$ in $L$. 
A virtual endomorphism $\varphi:M\ra L$ is called \textbf{simple}
if there are no non-trivial ideals of $L$ that are $\varphi$-invariant. 
Moreover, we say that $L$ is 
\textbf{self-similar of index} $\bm{d}$  if  $L$ 
admits a simple virtual endomorphism of index $d$.
Using Lazard's correspondence  \cite{Laz65}, which is
an isomorphism between the category of 
saturable pro-$p$ groups and the category of saturable Lie lattices over 
$\bb{Z}_p$, and results from \cite{GSKpsdimJGT},
we prove the following proposition.

\begin{propositionx} \label{pgsstlss2} 
Let $p$ be a prime and $k\in\bb{N}$.
Let $G$ be a saturable  $p$-adic analytic pro-$p$ group of dimension 
$\mr{dim}(G)\les p$, and $L_G$ be the $\bb{Z}_p$-Lie lattice associated with $G$. Then
$G$ is a self-similar group of index $p^k$
if and only if $L_G$ is a self-similar Lie lattice of index $p^k$.
\end{propositionx}

\vspace{0mm}

In \cite{GSKpsdimJGT}, Gonz\'alez-S\'anchez and Klopsch proved that
any torsion-free $p$-adic analytic pro-$p$ group $G$ of dimension
$\mr{dim}(G)<p$ is saturable, and that
such groups correspond bijectively to residually nilpotent $\bb{Z}_p$-Lie lattices.
Using this fact and the above proposition,
it is possible to classify such groups and give their self-similarity 
properties, provided that one is able to do so for the corresponding 
Lie lattices,
an easier task due to their linear nature.
With this in mind, we prove a sequence of theorems on $\bb{Z}_p$-Lie lattices 
for $p$ odd. First of all,
we give an explicit classification up 
to isomorphism of
3-dimensional unsolvable Lie lattices $L$ over $\bb{Z}_p$
({Theorem \ref{tsoluns}}),
complementing 
the classification of 3-dimensional solvable Lie lattices
provided in \cite{GSKpsdimJGT}. The same theorem determines exactly
which unsolvable Lie lattices are self-similar of index $p$. 
Next, we   
establish that, if $L\otimes_{\bb{Z}_p}\bb{Q}_p\simeq sl_2(\bb{Q}_p)$
then $L$ is self-similar of some index $p^k$, and we give estimates
for the {self-similarity index} of $L$
({Theorem \ref{tssindeta0}}); moreover, 
we prove the self-similarity
of some notable subalgebras of $sl_2(\bb{Z}_p)$ ({Theorem \ref{tsl2}}).
On the other hand, let $\bb{D}_p$
be a central simple $\bb{Q}_p$-division algebra of 
index 2, and let $sl_1(\bb{D}_p)$ 
be the set of elements of reduced trace zero in $\bb{D}_p$. We establish
that, if $L\otimes_{\bb{Z}_p}\bb{Q}_p\simeq sl_1(\bb{D}_p)$
then $L$ is not self-similar of index $p$ ({Theorem \ref{tsl1tnss}}).
We believe that such $L$'s are not self-similar of any index
(Conjecture \ref{conjsl1alg}).
Observe that $sl_2(\bb{Q}_p)$ and $sl_1(\bb{D}_p)$
represent the two isomorphism classes of 3-dimensional unsolvable 
$\bb{Q}_p$-Lie algebras (for $p$ odd).

Combining Proposition \ref{pgsstlss2} with our results on Lie lattices
we get the main result of the paper.

\begin{theoremx} \label{tmain3}
Let $p\ges 5$ be a prime, and fix $\rho\in\bb{Z}_p^*$ not a square modulo $p$. 
The following is a complete and irredundant
list of 3-dimensional unsolvable torsion-free $p$-adic analytic pro-$p$ groups, 
up to isomorphism: 
\begin{enumerate}
\item The pro-$p$ group 
$G_1(s_0, s_1, s_2, \vep_1, \vep_2)$, for  $0 \les s_0 < s_1 < s_2$ 
and $\vep_1,\vep_2\in\{0,1\}$,  associated with the
$\bb{Z}_p$-Lie lattice presented by:
$$ \langle\, x_0, x_1, x_2   \mid [x_1, x_2] =  p^{s_0}x_0,\,  [x_2, x_0] =
\rho^{\vep_1}p^{s_1}x_1,\, [x_0, x_1] =  \rho^{\vep_2}p^{s_2}x_2 \,\rangle .$$
\item  The pro-$p$ group $G_2(s_0, s_2, \vep_1)$, for  $1 \les s_0  < s_2$ 
and $\vep_1 \in\{0,1\}$,  associated with the $\bb{Z}_p$-Lie lattice 
presented by:
$$ \langle\, x_0, x_1, x_2   \mid [x_1, x_2] =  p^{s_0}x_0,\,  [x_2, x_0] =
-\rho^{\vep_1}p^{s_0}x_1,\, [x_0, x_1] =  p^{s_2}x_2\,\rangle .$$
\item  The pro-$p$ group $G_3(s_0, s_1, \vep_2)$, for  $0 \les s_0 < s_1 $
and $\vep_2\in\{0,1\}$,  associated with the $\bb{Z}_p$-Lie lattice 
presented by:
$$ \langle\, x_0, x_1, x_2   \mid [x_1, x_2] =  p^{s_0}x_0,\,  [x_2, x_0] =
p^{s_1}x_1,\, [x_0, x_1] = - \rho^{\vep_2}p^{s_1}x_2\,\rangle.$$
\item  The pro-$p$ group $G_4(s_0)$, for  $1 \les s_0 $,  associated with the 
$\bb{Z}_p$-Lie lattice
presented by:
$$ \langle\, x_0, x_1, x_2   \mid [x_1, x_2] =  
p^{s_0}x_0, \, [x_2, x_0] =  p^{s_0}x_1, \,[x_0, x_1] =  p^{s_0}x_2\,\rangle .$$
\end{enumerate}
Moreover, we have:
\begin{enumerate}
\item  $G_1(s_0, s_1, s_2, \vep_1, \vep_2)$ is not self-similar of index $p$.
\item $G_2(s_0, s_2, \vep_1)$ is self-similar of index $p$ if and only if
$\vep_1 = 0$.
\item $G_3(s_0, s_1, \vep_2)$ is self-similar of index $p$ if and only if
$\vep_2 = 0$.
\item $G_4(s_0)$ is self-similar of index $p$.
\end{enumerate}
\end{theoremx}
\vspace{3mm}

Let $p$ be an odd prime. Let ${\Delta}_p$ denote the (unique) maximal $\bb{Z}_p$-order in  $\bb{D}_p$ and write $\mathfrak{P}$ for the maximal ideal of ${\Delta}_p$.  Let $SL_1(\bb{D}_p)$ be the set of elements of reduced norm 1 
in $\bb{D}_p$ and, for  $m \ges 1$, let 
$SL_1^m({\Delta}_p): = SL_1(\bb{D}_p) \cap (1+ {\mathfrak{P}}^m)$. Any 3-dimensional unsolvable torsion-free $p$-adic analytic pro-$p$ group
is isomorphic to an open subgroup of exactly one of the groups $SL_1^1({\Delta}_p)$ and
$SL_2^{\sylow}(\bb{Z}_p)$, where $SL_2^{\sylow}(\bb{Z}_p)$ is a Sylow pro-$p$ subgroup
of $SL_2(\bb{Z}_p)$ (see \cite[Section 7.3]{GSKpsdimJGT}). 
The following two theorems present a dichotomy,
with respect to self-similarity,
among the open subgroups of $SL_2^{\sylow}(\bb{Z}_p)$ and of
$SL_1^1({\Delta}_p)$; moreover,
we believe that the ensuing conjecture is true.
The \textbf{self-similarity index} $\sigma(G)$  of a self-similar pro-$p$ group $G$ is defined to be the least power of $p$, say $p^k$,  such that $G$ is self-similar of index $p^k$. 

\begin{theoremx} \label{tmainsl2}
Let $p$ be a prime. 
\begin{enumerate}
\item 
 \begin{enumerate}
 \item For $p\ges 5$, any open subgroup of $SL_2^{\sylow}(\bb{Z}_p)$
 is self-similar. 
 \item For $p\ges 3$, any open subgroup of 
 $SL_2^1(\bb{Z}_p)$ is self-similar.
 \end{enumerate}
The self-similarity index of each of these groups $G$ is estimated in 
Table \ref{table1} (page \pageref{table1}) 
through the $\bb{Z}_p$-Lie lattice $L_G$ of $G$.
\item 
 \begin{enumerate}
 \item For $p\ges 5$, the group
 $SL_2^{\sylow}(\bb{Z}_p)$ and the terms $\gamma_k(SL_2^{\sylow}(\bb{Z}_p))$,
 $k\ges 1$,
 of its lower central series
 are self-similar of index $p$. 
 \item For $p\ges 3$, the congruence subgroups
 $SL^k_2(\bb{Z}_p)$ of $SL_2(\bb{Z}_p)$, $k\ges 1$, are self-similar of index $p$.
 \end{enumerate}
\item For $p\ges 5$, if $N$ is a non-trivial closed normal subgroup of 
$SL_2^{\sylow}(\bb{Z}_p)$ 
then $N$ is open in $SL_2^{\sylow}(\bb{Z}_p)$ 
and it is self-similar of index $p$ or $p^2$.
\end{enumerate}
\end{theoremx}

We also prove that if $p\ges 3$ and $G$ is a compact $p$-adic analytic group 
whose associated
$\bb{Q}_p$-Lie algebra is isomorphic to $sl_2(\bb{Q}_p)$ then $G$ is self-similar
(Corollary \ref{ccomppadic}).

\begin{theoremx} \label{tmainsl1}
Let $p$ be a prime.
\begin{enumerate}
\item If $p\ges 5$ then no open 
subgroup of $SL_1^1(\Delta_p)$ is self-similar of index $p$.
\item If $p\ges 3$ then no open 
subgroup of $SL_1^2(\Delta_p)$ is self-similar of index $p$.
\end{enumerate}
\end{theoremx}

\begin{conjx} \label{csl1}
For all primes $p\ges 5$, 
no open subgroup of $SL_1^1(\Delta_p)$
is self-similar.
For all primes $p\ges 3$, 
no open subgroup of $SL_1^2(\Delta_p)$
is self-similar.
\end{conjx} 

Observe that, in order to prove the conjecture, it suffices to prove
that $SL_1^1(\Delta_p)$ is not self-similar for $p \ges 5$
and that  $SL_1^2(\Delta_3)$ is not self-similar (cf. Corollary
\ref{cHssthenGss}).  Self-similar actions of 3-dimensional 
solvable torsion-free $p$-adic analytic pro-$p$ groups 
will be treated in an upcoming paper.

\vspace{3mm}

\noindent
\textbf{General notation.}
The set of natural numbers $\bb{N}=\{0,1,2,...\}$ is assumed to contain $0$.
If $R$ is a commutative ring ($\bb{Z}_p$ or $\bb{Q}_p$, in this paper)
we denote by $gl_n(R)$ the set of $n\times n$ matrices with coefficients in
$R$, and by $GL_n(R)$ the subset of matrices that are invertible over $R$. 
A square matrix with coefficients in $R$ is called
\textit{non-degenerate} if its determinant is not zero
(regardless whether the matrix is invertible over the given ring or not).
We write $\mr{diag}(a_1,...,a_n)$ for a diagonal matrix
with diagonal entries $a_1,...,a_n$. 
For the commutator of two elements $x,y$ of a group $G$,
we use the convention $[x,y] = x^{-1}y^{-1}xy$. 
For the lower central series we use the convention
$\gamma_0(G) = G$ and $\gamma_{k+1}(G) = [G,\gamma_k(G)]$.

\vspace{3mm}

\noindent
\textbf{Aknowledgements.}
The first author would like to thank Jos\'e Ram\'on-Mar\'i for suggesting
a useful reference.
The second author
thanks the Heinrich Heine University in D\"usseldorf for its hospitality.
The authors also thank the referee for useful comments.


\section{Self-similar groups}\label{sssgps}

Let $d\ges 1$ be an integer, and let $X:= \{0,...,d-1\}$, a finite set with
$d$ elements. 
We denote by $T_d$ the (regular) rooted tree associated
with the alphabet $X$. The vertices of
$T_d$ are the finite words in $X$; with a slight abuse of notation, we
write $v\in T_d$ when $v$ is such a word. There is an edge of $T_d$
exactly between vertices of type $v$ and $vx$,
for $v\in T_d$ and $x\in X$. The empty word is the root of $T_d$.
The automorphism group $\mr{Aut}(T_d)$ has a natural topology
that makes it profinite;  
the $n$-th level stabilizers of $T_d$ 
form a basis of open neighborhoods of the identity.
With any $g\in\mr{Aut}(T_d)$ and any $v\in T_d$
there is an associated restriction $g_{|v}\in\mr{Aut}(T_d)$.
For 
a left \textbf{action} $G\times T_d\rar T_d$ of a group $G$ on $T_d$ by rooted tree
isomorphisms, 
we use the notation $g\cdot v$ for the result of the action of $g\in G$ on 
$v\in T_d$.
An action is said to be
\textit{transitive on the first level} (respectively, \textit{level transitive}) 
if for all $x,y\in X$
(respectively, for all $v,w\in T_d$
of the same length)
there exists $g\in G$ such that $y= g\cdot x$ (respectively, $w= g\cdot v$).
An action is said to be \textbf{self-similar} if 
for all $g\in G$ and $x\in X$ there exists $h\in G$ such that for all 
$v\in T_d$ we have $ g\cdot xv = (g\cdot x)(h\cdot v)$,
where the right-hand side is a juxtaposition of words.
For a comprehensive study of self-similar groups the reader may consult,
for instance, \cite{NekSSgrp}.

Let $p$ be a prime number.
Observe that if $G$ is a self-similar pro-$p$ group of index $d$ then 
$d$ is a power of $p$ (Proposition \ref{pssiffve} and \cite[Lemma 1.18]{DixAnaProP}).
If $G$ is a nontrivial pro-$p$ group, it is natural to ask whether $G$
is self-similar of index $p$ (or just whether it is self-similar), and to
compute or estimate its self-similarity index.
It is known  that these questions 
can be tackled by looking at virtual 
endomorphisms (see Proposition \ref{pssiffve}).
We state the next definition and lemma (whose proof is left to the reader)
in the set theoretical setting, since these notions will be applied
to $\bb{Z}_p$-Lie lattices as well. 

\begin{definition}\label{dinvsubs2}
Let $B$ be a subset of a set $A$, and $\varphi:B\rar A$ be a function.
A subset $C\subseteq A$ is said to be 
$\varphi$\textbf{-invariant} if and only if
$C\subseteq B$ and $\varphi(C)\subseteq C$.  
We define the domain of the powers of $\varphi$
by $D_0:= A$, and by 
$D_{n+1}:=\{x\in B: \varphi(x)\in D_n\}$ for $n\ges 0$; 
we denote by $D_\infty$ the intersection of all of these domains. 
\end{definition}

\begin{lemma}\label{linvtdinf}
Let $\varphi:B\rar A$ be a function where $B\subseteq A$, let 
$C\subseteq A$, and let $D_\infty$ be the intersection of the domains 
of the powers of $\varphi$. If
$C$ is $\varphi$-invariant then $ C\subseteq D_\infty$.
\end{lemma}

\begin{proposition}\label{pssiffve} 
Let $G$ be a group and $d\ges 1$ be an integer.
Then $G$ is self-similar of index $d$ if and only if 
$G$ admits a simple virtual endomorphism of index $d$.
\end{proposition}

\begin{proof}
We sketch the proof. First, assume that $G$ admits
a faithful self-similar action on $T_d$ that is transitive on the first level.
Define $D$ to be the stabilizer of $0$ in $G$, and define $\varphi:D\rar G$
by $\varphi(g): = g_{|0}$ (cf. \cite[Section 2.5.2]{NekSSgrp}).
Then $G$ is a simple virtual endomorphism of index $d$.
On the  other hand, let $\varphi:D\rar G$
be a simple virtual endomorphism of index $d$, and let $X=\{0,...,d-1\}$.
Choose representatives $f_0,...,f_{d-1}$ of the left cosets of $D$ in $G$,
and assume $f_0=1$. We define maps 
$G\times X\rar X$, $(g,i)\mapsto g\cdot i$, and $G\times X\rar G$, $(g,i)\mapsto g_{|i}$
as follows. Given $g$ and $i$, there exists a unique $j$ such that $f_j^{-1}gf_i\in D$;
define $g\cdot i = j$ and $g_{|i}=\varphi(f_j^{-1}gf_i)$. These two maps
define a reversible automaton, and with any initial state $g\in G$, there is an 
associated automorphism $\Psi_g\in\mr{Aut}(T_d)$
(cf. \cite[Section 1.3]{NekSSgrp}). One proves that
the function $\Psi:G\rar \mr{Aut}(T_d)$ so obtained is a faithful
self-similar action of $G$ on $T_d$ that is transitive on the first level
(cf. \cite[Theorem 3.1]{NSAutBinTree} and 
\cite[Proposition 2.7.4]{NekSSgrp}). Hence, $G$ is self-similar of index $d$.
\end{proof}

\begin{corollary}\label{cHssthenGss}
Let $G$ be a group and $d\ges 1$ be an integer. Let $H$ be a subgroup
of $G$ of finite index. If $H$ is self-similar of index $d$ then $G$
is self-similar of index $d\,[G:H]$.
\end{corollary}

\vspace{3mm}

One may further specialize the problem of self-similarity by considering
torsion-free $p$-adic analytic pro-$p$ groups.
For a comprehensive study of $p$-adic analytic pro-$p$ groups,
the reader may consult \cite{DixAnaProP}.
In low dimension matters are easy since such groups $G$ may be 
classified. 

\begin{proposition} \label{pdim2}
Let $p$ be a prime, $k \ges 1$ be an integer, and 
$G$ be a non-trivial 
torsion-free $p$-adic analytic pro-$p$ group with $\mr{dim}(G) \les 2$. 
Then $G$ is self-similar of index $p^k$.
\end{proposition}

\begin{proof2}
For each $G$, we exhibit a simple virtual endomorphism 
$\varphi:D\rar G$ of index $p^k$ and apply Proposition \ref{pssiffve}.
If  $\mr{dim}(G) =1$ then $G$  is isomorphic to $\bb{Z}_p$.  We define  
$\varphi: p^k\bb{Z}_p\rar \bb{Z}_p$
by $\varphi(a)= p^{-k}a$; the
simplicity of $\varphi$
follows from Lemma \ref{linvtdinf}
after observing that $D_\infty = \{0\}$.

The case $\mr{dim}(G) =2$ is richer.
The groups under consideration are classified 
as follows (see \cite[Propositions 7.1, 7.2]{GSKpsdimJGT}). 
Assume that $s\in\bb{N}\cup \{\infty\}$
and that $s\ges 1$ if $p\ges 3$, and $s\ges 2$ if $p=2$.
We define $G_+(s)$ for all $p$, and $G_-(s)$ for $p=2$ through
the following pro-$p$ presentations:
$$
G_+(s) = \langle x, y\, |\, [x,y]= x^{p^s}\rangle
\qquad\mbox{and}\qquad
G_-(s) = \langle x, y\, |\, [x,y]= x^{-2-2^s}\rangle,
$$
where, by convention, $p^\infty =0$.
In order to give the most uniform treatment as possible, 
we consider the presentation
$G=\langle x, y \, | \, [x,y]= x^{u-1}\rangle$
where $u$ takes values $u = 1+p^s$ for all $p$, and $u=-1-2^s$ for $p=2$.
We observe that any element of $G$ is of the form $y^\beta x^\alpha$
for unique $\alpha,\beta \in \bb{Z}_p$. Moreover, 
multiplication in $G$ is given by $(y^\beta x^\alpha)(y^\delta x^\gamma)
= y^{\beta + \delta}x^{\alpha u^{\delta}+\gamma}$.
We divide the proof of self-similarity in three cases, according to
whether $u=1$, $u\neq \pm 1$ or $u=-1$.
Angle brackets below mean `generated closed subgroup'.
\begin{enumerate}
 
 \item \textit{Case} $G=G_+(\infty) =  \bb{Z}_p\times \bb{Z}_p$. 
 Let $D = \langle x^{p^k},y\rangle$ and consider the homomorphism $\varphi:D\rar G$ given by 
 $\varphi(x^{p^k})= y$ and $\varphi(y) = x$.
 Then $D_\infty=\{1\}$, so $\varphi$ is simple by
 Lemma \ref{linvtdinf}.
 
 \item \textit{Cases} $G=G_+(s)$ and $G=G_-(s)$,  
 with $s\in \bb{N}$.  
 Let $D = \langle x^{p^k},y\rangle$ and consider the homomorphism $\varphi:D\rar G$ given by 
 $\varphi(x^{p^k})= x$ and $\varphi(y) = y$.
 Then $D_\infty=\langle y\rangle$.  Let $N$ be a non-trivial normal subgroup 
 of $G$ and let $1 \neq  y^\beta x^\alpha  \in N$.  If $\alpha \neq 0$, then $y^\beta x^\alpha \not\in D_\infty$.
 If $\alpha =0$ then $\beta \neq 0$, so $z = xy^\beta x^{-1}=
 x^{u^\beta-1}\in N$, where $u=1+p^s$ or $u =-1-p^s$ (according to the case).
 Since $u\neq \pm 1$ and $\beta\neq 0$, we have $u^\beta -1\neq 0$, 
 so  $z\not\in D_\infty$. By Lemma \ref{linvtdinf}, $N$ is not $\varphi$-invariant. Hence, $\varphi$ is simple. 

 \item \textit{Case} $G=G_-(\infty)$ for $p=2$.
 Let $D = \langle x, y^{2^k}\rangle$ and consider the homomorphism $\varphi:D\rar G$ given by 
 $\varphi(x)= y$ and $\varphi(y^{2^k}) = y$. 
 We observe that the multiplication in $D$ 
 reads $(y^{2^k\beta} x^\alpha)(y^{2^k\delta} x^\gamma) 
 = y^{2^k(\beta + \delta)}x^{\alpha +\gamma}$, so that 
 $D\simeq \bb{Z}_2\times \bb{Z}_2$.  Assume, by contradiction, 
 that there exists a non-trivial $\varphi$-invariant normal subgroup $N$ of $G$.
 Every element of $N$ is of the form $y^{2^k\beta}x^\alpha$, 
 for some $ \alpha,\beta \in \bb{Z}_2$.  Let $1 \neq z = y^{2^k\beta}x^\alpha  \in N$.
 If $\alpha =0$, then $\beta \neq 0$, so $ 1  \neq y^{2^k\beta}  \in N$.
 On the other hand, if $\alpha \neq 0$
then  $y^{-1}(y^{2^k\beta}x^\alpha)y =
 y^{2^k\beta}x^{-\alpha}\in N$, and consequently, $(y^{2^k\beta}x^{\alpha})
 (y^{2^k\beta}x^{-\alpha})^{-1} =x^{2\alpha}\in N$. Hence, $\varphi(x^{2\alpha})= y^{2\alpha}\in N$. 
 Let $t\in \bb{N}$ be the minimum 2-adic valuation of a $\gamma \in \bb{Z}_2$
 such that $y^\gamma \in N$ (observe that $t\ges k$).   
 Then $\varphi(y^\gamma)=
 y^{2^{-k}\gamma}\in N$ and $v_2(2^{-k}\gamma) = t-k\in \bb{N}$, contradicting
 the minimality of $t$. \hfill $\square$
 \end{enumerate}
\end{proof2}

\begin{corollary}\label{cpadic2ss}
Let $p$ be a prime and $G$ be a compact $p$-adic analytic group of dimension
$\mr{dim}(G)\les 2$. Then $G$ is self-similar.
\end{corollary}

\vspace{5mm}

Since $p$-adic analytic pro-$p$ groups are topologically 
finitely generated, by a result of Serre their topology is determined by the group
structure (see \cite[Corollary 1.21]{DixAnaProP};
see also \cite{NSfgprof1}). However, this is not true in general for profinite groups that are not finitely generated. Modifying the definition to cover 
this case, 
we say that a \textit{profinite group} $G$  is \textit{self-similar of index} $d$ 
if $G$ admits a faithful self-similar action on $T_d$ that is transitive 
on the first level  and moreover the associated group homomorphism
$\Psi: G\rar \mr{Aut}(T_d)$ is continuous.
The next proposition is a topological analogue of 
Proposition \ref{pssiffve}; we leave its
proof  to the reader.

\begin{proposition}\label{pssiffvetop}
Let $G$ be a profinite group and $d\ges 1$ be an integer. Then
$G$ is self-similar of index $d$ (as a profinite group)
if and only if there exists a continuous simple virtual endomorphism
$\varphi: D\rar G$ of index $d$ with $D$ open in $G$.
\end{proposition}

We close the section with a technical lemma which 
is used to show that the actions
considered in Remark \ref{rleveltrans}
are level-transitive and not just transitive on the first level.

\begin{lemma}\label{lphiregular}
Let $p$ be a prime, and let $G$ be a group such that the index of any
finite-index subgroup of $G$ 
is a power of $p$. 
Let $\varphi: D\rar G$
be a virtual endomorphism of $G$ of index $p$, and let $(D_n)_{n\in\bb{N}}$ be
the sequence of the domains of the powers of $\varphi$.
If $\varphi(D_{n+1})\not\subseteq D_{n+1}$ for all $n\ges 0$, then
$\varphi$ is regular, namely, $[D_n:D_{n+1}]=p$ for all $n\ges 0$.
\end{lemma}

\begin{proof} Let $d_n := [D_n:D_{n+1}]$ for all $n\ges 0$.
Fix $n\ges 0$ for a while, and let $\varphi_n: D_{n+1}\rar D_n$ be the restriction
of $\varphi$. Then $D_{n+2} = \varphi_n^{-1}(D_{n+1})$ and
$[D_{n+1}:D_{n+2}] = [\varphi(D_{n+1})D_{n+1}:D_{n+1}]$.
Observe that the right-hand side of the last equality is the cardinality of the
set of
left cosets of $D_{n+1}$ in $D_n$ that are included in $\varphi(D_{n+1})D_{n+1}$,
so that $[\varphi(D_{n+1})D_{n+1}:D_{n+1}]\les [D_n:D_{n+1}]$.
It follows that $d_{n+1}\les d_n$. Since $d_0 = p$, one may use induction
on $n$ to prove that $d_n\les p$ and that $[G:D_n]<\infty$, for all
$n\ges 0$. Now, since $\varphi(D_{n+1})\not\subseteq D_{n+1}$, we have  $d_{n+1}>1$ and, consequently,
 $d_n = p$ for all $n\ges 0$.
\end{proof}


\section{Self-similar Lie lattices}\label{secliealg}

Let $p$ be a prime number. We take $\bb{Z}_p$
as the coefficient ring of our modules and algebras, and
denote by $v_p:\bb{Z}_p\rar \bb{N}\cup\{\infty\}$ the $p$-adic
valuation. 
Our main interest is on Lie lattices, 
namely, Lie algebras whose
underlying module is a lattice (i.e., a finite-dimensional free module). 
At several points, we find more elegant
and convenient to work with antisymmetric algebras. 
The ``bracket'' of two elements in such algebras is denoted by $[x,y]$.
Let $L$ be an $n$-dimensional antisymmetric $\bb{Z}_p$-algebra, $n\in\bb{N}$. 
We say that $L$ is \textbf{self-similar} if it is
self-similar of index $d$ for some positive integer $d$. The \textbf{self-similarity index} $\sigma(L)$  of a self-similar $L$ is defined 
to be the least positive integer  $d$ such that $L$ is self-similar of index $d$. 
We say that $L$ is just infinite if 
any non-zero ideal of $L$ has dimension $n$.
Moreover, we say that $L$ is hereditarily just infinite
if any $n$-dimensional subalgebra of $L$ is just infinite.
For the lower central series of $L$, we use the convention
$\gamma_0(L) = L$ and $\gamma_{k+1}(L)=[L, \gamma_k(L)]$.

\smallskip
A statement similar to Corollary \ref{cHssthenGss} 
holds for Lie lattices.

\vspace{-2mm}

\begin{lemma}\label{lsubss}
Let $L$ be an $n$-dimensional antisymmetric $\bb{Z}_p$-algebra and $M\subseteq L$ be
a subalgebra of dimension $n$. Moreover, suppose that $M$ is self-similar of index
$p^k$. Then $L$ is self-similar of index $p^{k}[L:M]$.
\end{lemma}

\begin{remark}\label{remdim2alg} 
Every non-trivial 
$\bb{Z}_p$-Lie lattice  $L$  of dimension $n \les 2$ is self-similar  of index
$p^k$ for all $k\ges 1$. 
We show this fact for $n=2$; 
the case $n=1$ has already been treated for $\bb{Z}_p$
as a group. If $L$  is a $2$-dimensional $\bb{Z}_p$-Lie lattice, 
then there is some  $s\in\bb{N}\cup \{\infty\}$ such that $L$ is isomorphic to
$L(s) = \langle x, y \,|\, [x,y]= p^sx\rangle$ 
(see, for instance, \cite[Section 7.1]{GSKpsdimJGT}), 
where  $p^\infty: = 0$. Consider the subalgebra 
$M = \langle p^k x,y\rangle$ of $L(s)$.  
For $s=\infty$ we define a homomorphism $\varphi: M\rar L(s)$ by
$\varphi(p^kx)= y$ and $\varphi(y) = x$,
while for $s\in \bb{N}$ we define a homomorphism 
$\varphi: M\rar L(s)$ by $\varphi(p^k x)= x$ and $\varphi(y) = y$. 
It is not difficult to see that  $\varphi$ is a simple virtual 
endomorphism of $L$ of index $p^k$. Note that the virtual endomorphisms 
in this example are exactly those of Proposition \ref{pdim2} 
(except for the extra groups that appear for $p = 2$).
\end{remark}

Next we deal with 3-dimensional $\bb{Z}_p$-Lie lattices. 

\begin{lemma} Let $L$ be a 3-dimensional $\bb{Z}_p$-lattice,
and let $\bm{x}=(x_0,x_1,x_2)$ be a basis of $L$.

\begin{enumerate} \label{lantbrmat}
\item The following formulas establish a bijection
(which depends on $\bm{x}$)
between antisymmetric brackets on $L$ and $3 \times 3$ matrices $A$ with
coefficients in $\bb{Z}_p$:
$$
\left\{
\begin{array}{lcl}
[x_1, x_2] & = & \sum_{i=0}^2 A_{i0} x_i \\[3pt]
[x_2, x_0] & = & \sum_{i=0}^2 A_{i1} x_i\\[3pt]
[x_0, x_1] & = & \sum_{i=0}^2 A_{i2} x_i. 
\end{array}
\right.
$$
\item Assume that an antisymmetric bracket over $L$ is given,
and let $A$ be its matrix with respect to $\bm{x}$.
Let $M$ be 3-dimensional submodule of $L$, and 
$\bm{y} = (y_0,y_1, y_2)$ be a basis of $M$
(the case $M=L$, change of basis, is included). Let $U$ be
the matrix of $\bm{y}$ with respect to $\bm{x}$, namely,
$y_j = \sum_i U_{ij} x_i$. Since $\mr{det}(U)\neq 0$, 
the formula
\begin{equation}\label{ebchbasis}
B = \mr{det}(U) U^{-1} A (U^{-1})^T
\end{equation}
defines a matrix $B\in gl_3(\bb{Q}_p)$. The following properties hold.
 \begin{enumerate}
 \item $M$ is a subalgebra of $L$ if and only if
 $B$ has coefficients in $\bb{Z}_p$. 
 \item If $M$ is a subalgebra then $B$ is the matrix
 of $M$ with respect to $\bm{y}$.
 \end{enumerate}
\end{enumerate}
\end{lemma}

\begin{proof} The proof is straightforward.
The reader may want to consult
\cite[page 13]{JacLieA} for similar arguments.
\end{proof}

\vspace{0mm}

\begin{definition}\label{ddiagbasis3}
Let $L$ be a 3-dimensional antisymmetric $\bb{Z}_p$-algebra, and let $\bm{x}$ be
a basis of $L$. Denote by $A$ the matrix of $L$ with respect to $\bm{x}$.
We say that $\bm{x}$ is \textbf{diagonalizing} if and only if
$A$ is diagonal, say $A = \mr{diag}(a_0,a_1,a_2)$.
We say that $\bm{x}$ is \textbf{well diagonalizing}
if and only if it is diagonalizing and $v_p(a_0)\les v_p(a_1) \les v_p(a_2)$ 
(the case $v_p(a_i)=\infty$ is not excluded).
\end{definition}

\begin{remark}\label{rdiagbasis}
A 3-dimensional antisymmetric $\bb{Z}_p$-algebra $L$ 
admits 
a diagonalizing basis if and only if it admits a well diagonalizing one.
This follows from the following more general observation.
Assume that $L$ admits a diagonalizing basis 
$\bm{x}$ and that 
$A= \mr{diag}(a_0,a_1,a_2)$ is the matrix of 
$L$ with respect to $\bm{x}$. Through 
a ``diagonal'' change of basis, one can make an arbitrary permutation
of the $a_i$'s, and multiply all the $a_i$'s by the same invertible
element of $\bb{Z}_p$.
\end{remark}

\begin{remark} \label{rsinv}
Let $L$ be a 3-dimensional antisymmetric $\bb{Z}_p$-algebra. Then there exist 
$s_0,s_1,s_2\in \bb{N}\cup\{\infty\}$ such that 
$L/[L,L]\simeq \bigoplus_{i=0}^2 (\bb{Z}/p^{s_i}\bb{Z})$ 
(where $p^\infty:=0$). The exponents $s_0,s_1,s_2$ of the invariant factors
will be called the $\bm{s}$\textbf{-invariants} 
of $L$ (they are all finite exactly when $L$ is unsolvable). 
Suppose that $L$ admits a diagonalizing basis and let
 $A=\mr{diag}(a_0,a_1,a_2)$ be the associated matrix.
Then the valuations $v_p(a_i)$ for $i\in\{0,1,2\}$
are the $s$-invariants of $L$.
\end{remark}

\begin{proposition}\label{p3dimins}
Let $L$ be a 3-dimensional antisymmetric $\bb{Z}_p$-algebra. The following holds.
\begin{enumerate}
\item \label{liliffsnzd2}
$L$ is an unsolvable Lie lattice if and only if the
matrix of $L$ with respect to some (equivalently, any) basis
of $L$ is symmetric and non-degenerate.
\item \label{p3diminsItem2}
Assume that $p\ges 3$ and that $L$ is, moreover, an unsolvable
Lie lattice. Then
 \begin{enumerate}
 \item \label{lexdiagbas2} $L$ admits a well diagonalizing basis.
 \item \label{l3dimhji2} $L$ is hereditarily just infinite.
 \end{enumerate}
\end{enumerate}
\end{proposition}
\begin{proof}
For part (1), symmetry
is related to the Jacobi identity (cf. \cite[page 13]{JacLieA}), while for 
the relation between non-degeneracy and unsolvability we argue as follows.
Observe that $L$ is 
unsolvable if and only if $\mr{dim}[L,L]=3$. 
Moreover, if $\bm{x}=(x_0,x_1,x_2)$ is a basis of $L$ then $[L,L]$
is generated by $y_0=[x_1,x_2]$, $y_1=[x_2,x_0]$ and $y_2=[x_0,x_1]$.
By definiton, the matrix of $L$ with respect to $\bm{x}$ has the coordinates of
the $y_i$'s on its columns, so that the matrix is non-degenerate if and only
if $\mr{dim}[L,L]=3$.

Item (a) of part (2) follows from part (1),
Equation \ref{ebchbasis} of Lemma \ref{lantbrmat}, and  
the well known fact that, for $p\ges 3$, 
any symmetric matrix $A\in gl_3(\bb{Z}_p)$
is diagonalizable through a congruence
$A\,\mapsto\,  V^TAV$ with $V\in GL_3(\bb{Z}_p)$, 
see \cite{SerArithmetic}, for instance.
Finally, item (b) of part (2) 
follows from item (a) and the fact that any 3-dimensional unsolvable
$\bb{Z}_p$-Lie lattice that admits a diagonalizing basis is just
infinite.
\end{proof}
\vspace{-2mm}

\begin{proposition}\label{pisotind}
Let $L$ be a 3-dimensional unsolvable antisymmetric $\bb{Z}_p$-algebra 
and $M, N$ be 3-dimensional
subalgebras of $L$. If $M\simeq N$ then $[L:M]=[L:N]$.
\end{proposition}

\begin{proof}
There exist bases $(x_0, x_1, x_2)$ and $(y_0, y_1, y_2)$ of $L$ and $M$ respectively
such that $y_i = p^{k_i}x_i$ for $i=0,1,2$ and some $k_i\in \bb{N}$.
Let $k = k_0+k_1+k_2$, so that $[L:M] = p^k$.
By unsolvability, the ordered sets
$([x_1,x_2], [x_2,x_0], [x_0, x_1])$ and 
$([y_1,y_2], [y_2,y_0], [y_0, y_1])$ form bases of $[L,L]$ and $[M,M]$ respectively. 
The matrix of this basis of $[M,M]$
with respect to the one of $[L,L]$ is diagonal with entries
$p^{k_1+k_2}$, $p^{k_2+k_0}$ and $p^{k_0+k_1}$, thus 
$[[L,L]:[M,M]]=p^{2k}=[L:M]^2$.
It follows that $[M:[M,M]\,] = [L:M]\cdot [L:[L,L]\,]$ (and the same for $N$).
Since $M\simeq N$ implies that $[M:[M,M]]= [N:[N,N]]$, we have that  $[L:M]=[L:N]$.
\end{proof}

\vspace{5mm}
\vspace{-2mm}

The following conjecture is a generalization of the above property. We also include
a version for groups.
\vspace{-1mm}

\begin{conjecture}\label{conjindexLA}
Let $L$ be a just-infinite $\bb{Z}_p$-Lie lattice with $\mr{dim}(L)>1$, 
and let  $M, N$ be 
subalgebras of $L$ of dimension  $\mr{dim}(L)$. 
If $M\simeq N$ then $[L:M]=[L:N]$.
\end{conjecture}

\vspace{-2mm}

\begin{conjecture}\label{conjindexGP}
Let $G$ be a torsion-free just-infinite $p$-adic analytic pro-$p$ group
with $\mr{dim}(G)>1$, 
and let $H, K$ be open subgroups  of $G$.  
If $H\simeq K$ then $[G:H]=[G:K]$.
\end{conjecture}

\vspace{0mm}
Now we construct some simple virtual endomorphisms.
\vspace{-2mm}

\begin{lemma} \label{li3dssst}
Let $L$ be a 3-dimensional $\bb{Z}_p$-Lie lattice whose matrix
with respect to some basis $(x_0,x_1,x_2)$ is 
\scalebox{0.8}{$
\left[
\begin{array}{rrr}
a  & 0 & 0\\
0 & 0 & b \\
0 & b & 0
\end{array}
\right]
$}
with  $a,b\in \bb{Z}_p$ and $a,b\neq 0$.
Then $L$ is self-similar of index $p$.
\end{lemma}

\begin{proof} 
Consider the subalgebra   $M: = \langle x_0, px_1 , x_2\rangle$ of $L$ of index $p$. Note that the module homomorphism $\varphi$ given by the
assignements $x_0\mapsto x_0$,
$px_1\mapsto x_1$ and $x_2\mapsto px_2$ is in fact
a homomorphism of algebras.
The domain of the $n$-th power of $\varphi$ is 
$D_n =\langle x_0, p^nx_1 , x_2\rangle$ for $n\ges 0$; hence
$D_\infty=\langle x_0, x_2\rangle$.
Let $I\neq\{0\}$ be an ideal of $L$, and let $0 \neq z = c_0x_0+c_1x_1+c_2x_2 \in I$.    We claim that there is an element $w\in I$ such that $w\not\in D_\infty$. Indeed, if $c_1\neq 0$, take $w = z$. If $c_0\neq 0$, take $w = [z,x_1]$.
Finally, if $c_2\neq 0$, take $w = [[z, x_1], x_1]$. Hence, $I$ is not
$\varphi$-invariant, so that $\varphi$  
is a simple virtual endomorphism of index $p$.

\end{proof}

\vspace{-2mm}
\begin{lemma} \label{li3dssst1}
Let $p\ges 3$ be a prime, and
let $L$ be a 3-dimensional $\bb{Z}_p$-Lie lattice. Suppose that the matrix
of $L$ with respect to some basis $(x_0,x_1,x_2)$ is 
\scalebox{0.8}{$
\left[
\begin{array}{rrr}
a  & 0 & 0\\
0 & p^s & 0 \\
0 & 0 & -p^s
\end{array}
\right]$}
with $a\in \bb{Z}_p$, $a\neq 0$, and $s\in\bb{N}$.
Then $L$ is self-similar of index $p$.
\end{lemma}

\begin{proof}
Since $p$ is odd, the formulas $y_0 = 2x_0$,
$y_1= x_1-x_2$ and $y_2=x_1+x_2$ define a new basis of $L$
whose matrix has the format of the one of Lemma \ref{li3dssst} 
with $b=2p^s$.
\end{proof}

\vspace{3mm}

We close the section by recording some results needed in the sequel. 
The proof of the lemma is a routine calculation which is left to the reader.

\begin{lemma}\label{lgenGnL}
Let $L$ be a 3-dimensional $\bb{Z}_p$-Lie lattice that admits a diagonalizing basis. 
Let $s_0, s_1, s_2\in\bb{N}\cup\{\infty\}$ be the $s$-invariants of $L$,
with $s_0\les s_1\les s_2$. The following holds.
\begin{enumerate}
\item 
Let
$(x_0,x_1,x_2)$
be a well diagonalizing basis of $L$.
Then 
the lower central series of $L$ is given by
$$ \gamma_n(L) = \langle  p^{s_i^{(n)}}x_i:\,i=0,1,2 \rangle\qquad\qquad
(n\ges 0),$$
where the exponents $s_i^{(n)}\in \bb{N}\cup\{\infty\}$ are given by
$s_0^{(0)} = s_1^{(0)} = s_2^{(0)} = 0$ and
$$
\left\{
\begin{array}{l}
s_0^{(2m+1)} = (m+1)s_0 + ms_1\\
s_1^{(2m+1)} = ms_0 + (m+1)s_1\\
s_2^{(2m+1)} = ms_0 + ms_1 + s_2
\end{array}
\right.
\quad m\ges 0;
\qquad 
\left\{
\begin{array}{l}
s_0^{(2m)} = ms_0 + ms_1\\
s_1^{(2m)} = ms_0+ ms_1\\
s_2^{(2m)} = ms_0 + (m-1)s_1 + s_2
\end{array}
\right.
\quad m\ges 1.
$$
\item $L$ is residually nilpotent if and only if $s_1\ges 1$.
\end{enumerate}
\end{lemma}

\begin{proposition}\label{pinsresnil2}
Let $L$ be a 3-dimensional $\bb{Z}_p$-Lie lattice that admits a diagonalizing basis. 
Let $s_0=0$, $s_1 =1$ and $s_2 =1$ be the $s$-invariants of $L$,
and suppose that $I\subseteq L$ is an ideal, $k\in\bb{N}$
and $I\not\subseteq \gamma_k(L)$. Then $\gamma_k(L)\subseteq I$.
\end{proposition}

\begin{proof}
Let the variables $u_0$, $u_1$ and $u_2$ stand for elements 
of $\bb{Z}_p^*$; their values may vary from one formula to another.
From the assumptions and Remarks \ref{rdiagbasis} and \ref{rsinv},
there exists a basis $(x_0,x_1,x_2)$ of $L$ such that
$[x_1, x_2] = u_0x_0$, $[x_2, x_0] = pu_1x_1$ and 
$[x_0, x_1] = pu_2x_2$. From 
$I\not\subseteq \gamma_k(L)$,
there exists $w = a_0x_0+a_1x_1+a_2x_2\in I$ 
such that $w\not\in \gamma_k(L)$.
We treat the case where $k=2m$ is even, 
$v_p(a_0)\les m-1$ and $v_p(a_1),v_p(a_2)\ges m$; 
the other cases are left to the reader.
From Lemma \ref{lgenGnL} it follows that
$\gamma_{2m}(L) = p^m L$, 
so that we only need to show that $p^mx_i\in I$ for $i\in\{0,1,2\}$.
Since $v_p(a_0)\les m-1$, it is enough to show that
$a_0px_i\in I$ for $i\in\{0,1,2\}$.
From
$[[w,x_1],x_0]= a_0p^2u_1x_1 \in I$
it follows that $p^{m+1}x_1\in I$,
so that $a_1px_1\in I$. From 
$[[w,x_2],x_2] = a_0pu_0x_0 + a_1pu_1x_1\in I$
it follows that $a_0px_0 \in I$, hence $a_1x_0,a_2x_0 \in I$.
Finally, from 
$[w,x_2] = a_1u_0x_0 + a_0pu_1x_1\in I$ and
$[w,x_1] = a_2u_0x_0 + a_0pu_2x_2\in I$,
it follows that $a_0px_1,a_0px_2\in I$.
\end{proof}

\vspace{-2mm}
\subsection{Classification of 3-dimensional unsolvable Lie lattices}\label{scla3dill}

We classify 3-dimensional unsolvable Lie lattices over $\bb{Z}_p$
for $p$ an odd prime.
The classification is given in terms of $3\times 3$ matrices as explained below.
Observe that the \textbf{canonical forms} that we will actually use in the paper are
given in {Remark \ref{rreci3d}}.

With any matrix $A\in gl_3(\bb{Z}_p)$,
we associate a 3-dimensional antisymmetric $\bb{Z}_p$-algebra 
$L_A$ as follows: the underlying lattice is $\bb{Z}_p^3$ 
endowed with the canonical basis;
the bracket is induced by the matrix $A$ as in Lemma \ref{lantbrmat}.
One can see that $L_A\simeq L_B$ if and only if 
$A$ is \textbf{multiplicatively congruent} to $B$
(denoted $A\sim_{MC} B$), namely, there exist 
$u\in\bb{Z}_p^*$ and $V\in GL_3(\bb{Z}_p)$ such that
$B = u V^T A V$.

\vspace{0mm}

\begin{theorem}\label{tclai3d}
Let $p\ges 3$ be a prime, and fix $\rho\in\bb{Z}_p^*$ 
not a square modulo $p$. Then the antisymmetric
$\bb{Z}_p$-algebras $L_A$ associated with the diagonal matrices listed below
in four families
constitute a complete and irredundant list of
3-dimensional unsolvable Lie lattices over $\bb{Z}_p$.
\begin{enumerate}
\item  
$A = \mr{diag}(p^{s_0}, \rho^{\vep_1}p^{s_1},
\rho^{\vep_2}p^{s_2})$ with $0\les s_0 < s_1 < s_2$ and $\vep_1,\vep_2\in\{0,1\}$.
\item  
$A = \mr{diag}(p^{s_0}, \rho^{\vep_1}p^{s_0},
p^{s_2})$ with $0\les s_0 < s_2$ and $\vep_1\in\{0,1\}$.
\item 
$A = \mr{diag}(p^{s_0}, p^{s_1},
\rho^{\vep_2}p^{s_1})$ with $0\les s_0 < s_1$ and $\vep_2\in\{0,1\}$. 
\item  
$A = \mr{diag}(p^{s_0}, p^{s_0},
p^{s_0})$ with $0\les s_0$.
\end{enumerate}
\end{theorem}

\begin{proof}
We denote by $\Omega$ the set of matrices listed in the statement.
Observe that if $A,B\in gl_3(\bb{Z}_p)$,
$A\sim_{MC} B$ and $A$ is symmetric and non-degenerate, then
$B$ shares the same properties.
By Proposition \ref{p3dimins} and the observations preceding the statement of the theorem,
it suffices to show that $\Omega$ is a set of representatives for the symmetric
non-degenerate matrices of $gl_3(\bb{Z}_p)$ modulo
the relation $\sim_{MC}$ of multiplicative congruence. 
We write $A \sim_C B$ for the congruence relation, that is,
$B = V^TA V$ for some $V\in GL_3(\bb{Z}_p)$. Of course,
$A\sim_C B$ implies $A\sim_{MC} B$. The congruence classes
of symmetric non-degenerate matrices (i.e., non-degenerate 
quadratic forms)
are classified (see \cite[Theorem 3.1, page 115]{CasRQF}).
We can write the congruence classes in the form:
\begin{enumerate}
\item 
$A = \mr{diag}(\rho^{\vep_0}p^{s_0}, \rho^{\vep_1}p^{s_1},
\rho^{\vep_2}p^{s_2})$ with $0\les s_0 < s_1 < s_2$
and $\vep_0,\vep_1,\vep_2\in\{0,1\}$.
\item 
$A = \mr{diag}(p^{s_0}, \rho^{\vep_1}p^{s_0},
\rho^{\vep_2}p^{s_2})$ with $0\les s_0  < s_2$ and $\vep_1,\vep_2\in\{0,1\}$.
\item  
$A = \mr{diag}(\rho^{\vep_0}p^{s_0}, p^{s_1},
\rho^{\vep_2}p^{s_1})$ with $0\les s_0 < s_1$ and $\vep_0,\vep_2\in\{0,1\}$.
\item 
$A = \mr{diag}(p^{s_0}, p^{s_0},
\rho^{\vep_2}p^{s_0})$ with $0\les s_0$ and $\vep_2\in\{0,1\}$.
\end{enumerate}
In some sense, there are twice as many congruence classes as 
elements of $\Omega$: in the list of the congruence classes compared to the list of elements of $\Omega$, 
there is an extra factor $\rho^{\vep_i}$. We proceed as follows.

First of all, we show that any symmetric non-degenerate $B$ 
is multiplicatively congruent 
to some $A\in \Omega$.
Indeed, some congruence puts $B$ in one of the forms given above, and we
complete the argument for the second family,
namely, we assume that $B$ is congruent to 
$C = \mr{diag}(p^{s_0}, \rho^{\vep_1}p^{s_0},
\rho^{\vep_2}p^{s_2})$; the argument for the other families 
is similar. If $\vep_2 = 0$ we are done.
If $\vep_2 = 1$ then we multiply $C$ by $\rho$, 
getting
$B\sim_{MC}\mr{diag}(\rho p^{s_0}, \rho\rho^{\vep_1}p^{s_0},\rho^2 p^{s_2})$.
Applying \cite[Lemma 3.4, page 115]{CasRQF}, we ``discharge'' the factor
$\rho$ in the first diagonal entry to the second, which is done through
a congruence, so that 
$B\sim_{MC}\mr{diag}(p^{s_0},\rho^2\rho^{\vep_1}p^{s_0},\rho^2 p^{s_2})$.
Another (obvious) congruence eliminates the factors $\rho^2$ 
on the second and third diagonal entries,
getting one of the forms in the statement, namely,
$B\sim_{MC}\mr{diag}(p^{s_0}, \rho^{\vep_1}p^{s_0},p^{s_2})$.

The last thing to be proven is that no two distinct matrices
in $\Omega$ are multiplicatively congruent to each other. Formally, take $A,B\in \Omega$
and assume that they are multiplicatively congruent. We have to show that $A= B$.
First of all, observe that
no two distinct elements of $\Omega$ are congruent.
In other words, if we show that $A\sim_C B$ then we are done.
By assumption, there exist $u\in \bb{Z}_p^*$ and $V\in GL_3(\bb{Z}_p)$
such that $B = u V^TA V$. Hence, $B \sim_C uA$.
If $u$ is a square then $uA\sim_C A$ and we are done.
We now show that the other possibility, namely, that $u$ is not a square,
leads to a contradiction. Since $u=\rho v^2$
for some $v\in \bb{Z}_p^*$,
we see that $uA\sim_C \rho A$. Summarizing, $B\sim_C \rho A$,
where $A, B \in \Omega$.
Again, we should analyze four cases, depending on which family $A$ belongs to.
We will do the case of the second family, the others being similar.
The matrix $\rho A$ is 
$\mr{diag}(\rho p^{s_0}, \rho\rho^{\vep_1}p^{s_0},\rho p^{s_2})$,
whose congruence class is shown to be represented by 
$\mr{diag}(p^{s_0}, \rho^{\vep_1}p^{s_0},\rho p^{s_2})$
after discharging the first $\rho$ as above.
Hence, $B = \mr{diag}(p^{s_0}, \rho^{\vep_1}p^{s_0},\rho p^{s_2})$,
but this is a contradiction, since $ \mr{diag}(p^{s_0}, \rho^{\vep_1}p^{s_0},\rho p^{s_2})$  is not an element of $\Omega$.
\end{proof}

\begin{remark} \label{rreci3d}
The classification given in Theorem \ref{tclai3d} is a natural one
if one takes the classification of quadratic forms as a starting point.
On the other hand, in order to formulate the main theorem of
this section (Theorem
\ref{tsoluns}) in a more ``uniform'' way, we change some of the 
representative matrices as follows.
We recall that $\rho \in \bb{Z}_p^*$ is a fixed non-square modulo $p$,
and that $p$ is assumed to be odd.
The ``new'' matrices are:
\begin{enumerate}
\item 
$A = \mr{diag}(p^{s_0}, \rho^{\vep_1}p^{s_1},
\rho^{\vep_2}p^{s_2})$ with $0\les s_0 < s_1 < s_2$ and $\vep_1,\vep_2\in\{0,1\}$.
\item 
$A = \mr{diag}(p^{s_0}, -\rho^{\vep_1}p^{s_0},
p^{s_2})$ with $0\les s_0  < s_2$ and $\vep_1\in\{0,1\}$. 
\item  
$A = \mr{diag}(p^{s_0}, p^{s_1},
-\rho^{\vep_2}p^{s_1})$ with $0\les s_0 < s_1$ and $\vep_2\in\{0,1\}$.
\item 
$A = \mr{diag}(p^{s_0}, p^{s_0},
p^{s_0})$ with $0\les s_0$.
\end{enumerate}
Observe that we have just inserted two `minus' signs, one in family (2) and
the other in family (3).
We explain which kind of replacement has been performed by
treating family (2); the argument for family (3) is completely analogous. 
Recall that $-1$ is a square in $\bb{Z}_p^*$ if and only if $p\equiv 1$ modulo 4.
Below, on the left we write the representative with respect to the 
original classification,  while on the right
we write the corresponding representative in the new classification,
depending on the residue class of $p$ modulo 4. We have:
$$
\mr{diag}(p^{s_0}, p^{s_0},
p^{s_2}) \sim_{MC} 
\left\{
\begin{array}{ll}
\mr{diag}(p^{s_0}, -p^{s_0},
p^{s_2}) & \mbox{if }p\equiv 1\mbox{ modulo 4} \\[10pt]
\mr{diag}(p^{s_0}, -\rho p^{s_0},
p^{s_2}) & \mbox{if }p\equiv 3\mbox{ modulo 4} \\
\end{array}
\right.
$$
and
$$
\mr{diag}(p^{s_0}, \rho p^{s_0},
p^{s_2}) \sim_{MC}
\left\{
\begin{array}{ll}
\mr{diag}(p^{s_0}, -\rho p^{s_0},
p^{s_2}) & \mbox{if }p\equiv 1\mbox{ modulo 4} \\[10pt]
\mr{diag}(p^{s_0}, -p^{s_0},
p^{s_2}) & \mbox{if }p\equiv 3\mbox{ modulo 4}. \\
\end{array}
\right.
$$
As a matter of terminology, each of the matrices in the  enumeration above
will be called the \textbf{canonical matrix} of the corresponding 
isomorphism class of $\bb{Z}_p$-Lie lattices.
A basis of a lattice $L$ whose associated matrix is canonical
will be called a \textbf{canonical basis} of $L$, and the corresponding 
presentation (i.e., the commutation relations of the basis elements)
will be called the \textbf{canonical presentation} of $L$.
We will also use the terminology \textbf{canonical form} in order to refer
generically to the canonical matrix or to the canonical presentation.
\end{remark}

\begin{example}\label{ecansl}
We give the canonical matrix
of some relevant $\bb{Z}_p$-Lie lattices, where $p$ is assumed to be odd. 
Before doing that, we recall some definitions  from 
\cite[Section 2]{KloSL1}. Choose $\rho \in \{ 1, 2, ..., p-1\}$ not a square modulo $p$ and let 
$$\bb{D}_p = \bb{Q}_p + \bb{Q}_p\mb{u} + \bb{Q}_p\mb{v} + \bb{Q}_p \mb{uv}$$ 
 be the quaternion algebra defined by the multiplication rules 
$$\mb{u}^2 = \rho,\quad \mb{v}^2 = p, \quad\mb{uv} = -\mb{vu};$$
then $\bb{D}_p$ is a central division algebra of index 2 over $\bb{Q}_p$. For an element $\mb{z} = \alpha + \beta \mb{u} +\gamma \mb{v} + \delta \mb{uv} \in \bb{D}_p$, the reduced norm and reduced trace are given by 
$$ \textrm{N}(\mb{z}) =    \alpha^2 -  \rho \beta^2 - (\gamma^2 - \rho \delta^2)p  \quad\textrm{ and } \quad  
\textrm{T}(\mb{z}) = 2\alpha.$$  The ring of integers ${\Delta}_p =  \bb{Z}_p + \bb{Z}_p\mb{u} + \bb{Z}_p\mb{v} + \bb{Z}_p \mb{uv}$ of  $\bb{D}_p$ forms the unique maximal  $\bb{Z}_p$-order in $\bb{D}_p$, and  $\mathfrak{P} = \mb{v}{\Delta}_p$ is the maximal ideal of ${\Delta}_p$.  Let $SL_1(\bb{D}_p)$ be the set of elements of reduced norm 1 
in $\bb{D}_p$; this is a 3-dimensional compact $p$-adic analytic group. Note that  $SL_1(\bb{D}_p)$  coincides with the group $SL_1({\Delta}_p) = SL_1(\bb{D}_p)  \cap {\Delta}_p$.   For any integer $m\ges 1$, let 
$SL_1^m({\Delta}_p) = SL_1({\Delta}_p) \cap (1+ {\mathfrak{P}}^m)$ be the $m$-th congruence subgroup of $SL_1({\Delta}_p)$.  Similarly,  let $sl_1(\bb{D}_p)$ denote the set of elements of reduced trace zero in $\bb{D}_p$. Then  $sl_1({\Delta}_p) = sl_1(\bb{D}_p)  \cap {\Delta}_p$, and $sl_1^m({\Delta}_p) = sl_1({\Delta}_p) \cap {\mathfrak{P}}^m$ is the $m$-th congruence subalgebra of $sl_1({\Delta}_p)$.

On the right column of the list below, where $k\in \bb{N}$, 
$sl_2^k(\bb{Z}_p) = p^k sl_2(\bb{Z}_p)$ and $sl_1^0(\Delta_p)=sl_1(\Delta_p)$,
we display the canonical matrix of the Lie lattices appearing on the left.
\vspace{0mm}
$$
\begin{array}{rll}
\mbox{Lattice}
&&
\mbox{Matrix}
\\[5pt]
sl_2(\bb{Z}_p)
&-&
\mr{diag}(1,1,1)
\\[3pt]
sl_2^k(\bb{Z}_p)
&-&
\mr{diag}(p^k,p^k,p^k)
\\[3pt]
sl_2^{\sylow}(\bb{Z}_p)
&-&
\mr{diag}(1,p,-p)
\\[3pt]
\gamma_{2k}(sl_2^{\sylow}(\bb{Z}_p))
&-&
\mr{diag}(p^k,p^{k+1},-p^{k+1})
\\[3pt]
\gamma_{2k+1}(sl_2^{\sylow}(\bb{Z}_p))
&-&
\mr{diag}(p^{k+1},-p^{k+1},p^{k+2})
\\[3pt]
sl_1(\Delta_p)
&-&
\mr{diag}(1,-\rho,p)
\\[3pt]
sl_1^{2k}(\Delta_p)
&-&
\mr{diag}(p^k,-\rho p^k, p^{k+1})
\\[3pt]
sl_1^{2k+1}(\Delta_p)
&-&
\mr{diag}(p^k,p^{k+1},-\rho p^{k+1}).
\end{array}
$$
We prove the claims given in the list. Let $(x_0, x_1, x_2)$
be the basis
of $sl_2(\bb{Z}_p)$
given by
$
x_0 =
\begingroup 
\scriptsize
\left[
\begin{array}{rr}
1 & 0\\
0 & -1
\end{array}
\right]
\endgroup
$,
$x_1 =
\begingroup 
\scriptsize
\left[
\begin{array}{rr}
0 & 1\\
0 & 0
\end{array}
\right]
\endgroup
$ and
$x_2 =
\begingroup 
\scriptsize
\left[
\begin{array}{rr}
0 & 0\\
1 & 0
\end{array}
\right]
\endgroup
$.
The matrix of $sl_2(\bb{Z}_p)$ with respect to the
new basis
$(2^{-1}x_0, 2^{-1}(x_1+x_2), 2^{-1}(-x_1+x_2))$ is 
$\mr{diag}(1,1,-1)$, from which it follows that the $s$-invariants 
are $s_0=s_1=s_2=0$, so that the matrix of $sl_2(\bb{Z}_p)$
with respect to some basis $(y_0, y_1, y_2)$ is the identity
(see Remark \ref{rreci3d}). Since $sl_2^k(\bb{Z}_p) = p^k sl_2(\bb{Z}_p)$,
the matrix of $sl_2^k(\bb{Z}_p)$ with respect to the basis
$(p^ky_0, p^k y_1, p^k y_2)$ is $\mr{diag}(p^k,p^k,p^k)$.

Regarding $sl_2^{\sylow}(\bb{Z}_p)$, its
matrix with respect to the basis
$(z_0, z_1, z_2)=(2^{-1}px_0, 2^{-1}(x_1+px_2), 2^{-1}(-x_1+px_2))$
is $\mr{diag}(1,p,-p)$, as desired.
From Lemma \ref{lgenGnL}, we have:
$$
\left\{
\begin{array}{rcl}
\gamma_{2k}(L) &=& 
\langle  p^{k}z_0, p^k z_1, p^k z_2\rangle\\[5pt]
\gamma_{2k+1}(L) &=& 
\langle  p^{k}z_0, p^{k+1} z_1, p^{k+1} z_2\rangle.
\end{array}
\right.
$$
One computes that
the matrix of $\gamma_{2k}(L)$ with respect to the displayed basis
is the desired canonical form, while the matrix  
of $\gamma_{2k+1}(L)$ is
$\mr{diag}(p^{k+2}, p^{k+1}, -p^{k+1})$, which can be easily put 
into the desired canonical form (see Remark \ref{rdiagbasis}).

Finally, we consider $sl_1(\Delta_p)$ and its congruence subalgebras.
The ordered set  
$({x_0},{x_1},x_2) = (2^{-1}\mb{uv}, 2^{-1}\mb{v},2^{-1}\mb{u})$
is a basis of $sl_1(\Delta_p)$ over $\bb{Z}_p$, 
and the commutation relations are given by
$$
[x_1, x_2] = -x_0,\quad
[x_2, x_0] = \rho x_1, \quad
[x_0, x_1] = px_2.   $$ 
More generally,
for all $k\ges 0$, $(p^kx_0, p^kx_1, p^kx_2)$
is a basis of $sl_1^{2k}(\Delta_p)$, while 
$(p^kx_0, p^kx_1, p^{k+1}x_2)$
is a basis of $sl_1^{2k+1}(\Delta_p)$
(see \cite{KloSL1}). 
One may take these bases and apply 
\cite[Lemma 3.4, page 115]{CasRQF} to compute the desired canonical matrix
of the respective lattices.
\end{example}


\subsection{Non-self-similarity theorem}\label{snsscases}

\begin{definition}\label{dnonselfsc2}
Let $L$ be a 3-dimensional antisymmetric $\bb{Z}_p$-algebra, and let $\bm{x}$ be
a basis of $L$. Denote by $A$ the matrix of $L$ with respect to $\bm{x}$.
We say that 
$\bm{x}$
satisfies the \textbf{non-self-similarity condition}
if and only if $\bm{x}$ is well diagonalizing and, 
denoting $A=\mr{diag}(a_0,a_1,a_2)$, $Z_0 = \{0,...,p-1\}$ and $Z_1=\{1,...,p-1\}$,
the following conditions are satisfied:
 \begin{enumerate}
 \item $v_p(e^2a_0 +a_1) = v_p(a_0)$, for all $e\in Z_1$.
 \item $v_p(e^2a_0 +f^2 a_1 + a_2) = v_p(a_0)$, for all $e\in Z_1$ and $f\in Z_0$.
 \item $v_p(f^2a_1 +a_2) = v_p(a_1)$, for all $f\in Z_1$.
 \end{enumerate}
\end{definition}

\begin{theorem}\label{tsl11noss}
Let $p\ges 3$ be a prime, and let $L$ be a
3-dimensional unsolvable $\bb{Z}_p$-Lie lattice. Suppose that
there exists a basis of $L$ that
satisfies the non-self-similarity condition,
and that the $s$-invariants $s_0\les s_1\les s_2$ of $L$ 
are not all equal (i.e., $s_0<s_2$).
Then $L$ is not self-similar of index $p$.
\end{theorem}

\vspace{0mm}

Before proving the theorem above, we need some preparation.

\begin{lemma}\label{lstrongssc}
Let $L$ be a 3-dimensional unsolvable
$\bb{Z}_p$-Lie lattice that admits a diagonalizing basis, 
and let $\bm{x}$ be a 
well dia\-gonalizing basis of $L$.
Let $A = \mr{diag}(a_0,a_1,a_2)$ be the matrix of $L$ with respect 
to $\bm{x}$, and let $a_i = u_ip^{s_i}$ with $s_i\in\bb{N}$,
$u_i\in\bb{Z}_p^*$, and $s_0\les s_1\les s_2$.
Then the basis $\bm{x}$ satisfies the non-self-similarity condition
whenever one of the following conditions is satisfied:
\begin{enumerate}
\item $s_0<s_1<s_2$; or
\item $s_0 = s_1 < s_2$ and $-u_0u_1$ is not a square modulo $p$; or
\item $s_0 < s_1 = s_2$ and $-u_1u_2$ is not a square modulo $p$.
\end{enumerate}
\end{lemma}

\begin{proof}
 Clearly, the elements of $\{1,...,p-1\}$ are invertible in $\bb{Z}_p$.
 Note that some of the desired conditions of Definition
\ref{dnonselfsc2} follow directly from one of the 
basic properties of valuations: if $v_p(a)<v_p(b)$ then $v_p(a+b) = v_p(a)$.
Assume that (2) holds, and let $e\in Z_1$ and $f\in Z_0$. Then  $e^2u_0 + f^2u_1 + u_2p^{s_2-s_0}$
is invertible in $\bb{Z}_p$; indeed, if not, 
then reducing modulo $p$ we have $-u_0u_1 \equiv (u_1fe^{-1})^2$, 
which yields a contradiction, since by assumption 
$-u_0u_1$ is not a square modulo $p$. 
Now, $e^2a_0 +f^2 a_1 + a_2 = (e^2u_0 + f^2u_1 + u_2p^{s_2-s_0})p^{s_0}$ 
implies that $v_p(e^2a_0 +f^2 a_1 + a_2) = s_0 = v_p(a_0)$. 
This proves condition (2) of Definition \ref{dnonselfsc2}. 
The proof of the other cases is similar and is left to the reader.
\end{proof}

\begin{remark}\label{rcannssc}
For $p\ges 3$,
we consider the classification of 3-dimensional unsolvable $\bb{Z}_p$-Lie lattices
given in Remark \ref{rreci3d}. Observe that,
by assumption, the basis corresponding to each of the canonical forms 
is well diagonalizing.
One may apply Lemma \ref{lstrongssc} to see that a
canonical basis of a Lie lattice presented by any of the following canonical forms satisfies
the non-self-similarity condition:
\begin{enumerate}
\item  
$A = \mr{diag}(p^{s_0}, \rho^{\vep_1}p^{s_1},
\rho^{\vep_2}p^{s_2})$ with $s_0 < s_1 < s_2$.
\item 
$A = \mr{diag}(p^{s_0}, -\rho p^{s_0},
p^{s_2})$ with $s_0 < s_2$.
\item 
$A = \mr{diag}(p^{s_0}, p^{s_1},
-\rho p^{s_1})$ with $s_0 < s_1$.
\end{enumerate}
It will turn out that, among the 3-dimensional unsolvable ones,
these Lie lattices  are exactly the ones
that are not self-similar of index $p$.
\end{remark}

\begin{definition}\label{dzxiuxi}
$\phantom{a}$\\\vspace{-5mm}
\begin{enumerate}
\item Let $\Xi$ be the set of symbols $\xi$ of the form
$(), (e), (e,f)$, where $e,f\in Z_0:= 
\{0,1,...,p-1\}$.
We make a partition of the set $\Xi$ as follows,
where $Z_1:= \{1,...,p-1\}$:
\begin{enumerate}
\item $\Xi_0 =  \big\{()\big\}\cup \big\{(e):\, e\in Z_1\big\}\cup
\big\{(e,f):\, e\in Z_1,\, f\in Z_0\big\}$.
\item $\Xi_1 =  \big\{(0)\big\}\cup \big\{ (0,f):\, f\in Z_1\big\}$.
\item $\Xi_2  =  \{(0,0)\}$.
\end{enumerate}
\item With any $\xi\in \Xi$ we associate a matrix $U_\xi$:
\begin{equation*}
U_{()}=
\begingroup
\scriptsize
\left[
\begin{array}{rrr}
p & 0 & 0\\
0 & 1 & 0\\
0 & 0 & 1
\end{array}
\right]
\endgroup
\qquad
U_{(e)} =
\begingroup
\scriptsize
\left[
\begin{array}{rrr}
1 & 0 & 0\\
e & p & 0\\
0 & 0 & 1
\end{array}
\right]
\endgroup
\qquad
U_{(e,f)}=
\begingroup
\scriptsize
\left[
\begin{array}{rrr}
1 & 0 & 0\\
0 & 1 & 0\\
e & f & p
\end{array}
\right].
\endgroup
\end{equation*}
\item Let $L$ be a 3-dimensional $\bb{Z}_p$-lattice endowed with a basis
$\bm{x} = (x_0,x_1,x_2)$. We define a submodule  $L^\xi= \langle y_0, y_1, y_2\rangle$
of $L$ by 
$y_i =\sum_{j=0}^2(U_\xi)_{ji} x_j$, for $i=0,1,2$. 
\end{enumerate}
\end{definition}

\begin{lemma}\label{lclaipsmod}
Let $L$ be a 3-dimensional $\bb{Z}_p$-lattice endowed with a basis
$(x_0,x_1,x_2)$.
The map that associates with any $\xi\in \Xi$ the submodule $L^\xi$ of $L$
is a bijection from
$\Xi$ to the set of index-$p$ submodules of $L$.
\end{lemma}

\begin{proof}
Given $U\in gl_3(\bb{Z}_p)$, we define a triple $\bm{y}=(y_0,y_1,y_2)$
as in Definition \ref{dzxiuxi}
(and viceversa, any triple $\bm{y}$ corresponds to a unique matrix $U$). 
The submodule $M$ generated by $\bm{y}$ 
has dimension 3 if and only if $\mr{det}(U)\neq 0$;
in this case, $\bm{y}$ is a basis of $M$.
Moreover, $M$ has index $p$ in $L$ 
if and only if $v_p(\mr{det}(U))=1$.
Two  matrices $U$ and  $U'$ determine the same submodule if and only if
there exists $V\in GL_3(\bb{Z}_p)$ such that $U' = UV$.
Finally, note that Gaussian reduction along the columns of a matrix $U$ 
with $v_p(\mr{det}(U))=1$ reduces $U$ to some $U_\xi$
for exactly one value of $\xi$ (see 
\cite[Theorem 2.9 on page 302, Theorem 2.13 on page 304]{AWalgmod}
for details). 
\end{proof}

\begin{lemma}\label{lscip}
Let $L$ be a 3-dimensional $\bb{Z}_p$-Lie lattice that admits a 
diagonalizing basis $\bm{x}=(x_0,x_1,x_2)$, and
denote by  
$A =\mr{diag}(a_0,a_1,a_2)$ the corresponding matrix. 
For $\xi \in \Xi$, let $L^\xi$ be the index-$p$ submodule of $L$ 
of Definition \ref{dzxiuxi}, which comes endowed with a basis
$\bm{y}=(y_0,y_1,y_2)$. We denote by $B_\xi$ the associated matrix
$B$ of Equation
(\ref{ebchbasis}) of Lemma \ref{lantbrmat}.
Then the matrix $B_\xi$ is given as follows, where 
$e,f \in \{0,...,p-1\}$
and the vertical bars enhance readability:
$$ 
B_{()} =
\begingroup
\scriptsize
\left[
\begin{array}{r|r|r}
p^{-1}a_0 & 0 & 0 \\
0 & p a_1 & 0 \\
0 & 0 & pa_2 
\end{array}
\right]
\endgroup ,
\qquad\qquad
B_{(e)} =
\begingroup
\scriptsize
\left[
\begin{array}{r|r|r}
pa_0 & -ea_0 & 0 \\
-ea_0 & p^{-1}(e^2a_0+a_1) & 0 \\
0 & 0 & pa_2 
\end{array}
\right]
\endgroup ,
$$
$$ 
B_{(e,f)} =
\begingroup
\scriptsize
\left[
\begin{array}{r|r|r}
pa_0 & 0 & -ea_0 \\
0 & p a_1 & -f a_1 \\
-ea_0 & -f a_1 & p^{-1}(e^2a_0+f^2 a_1+ a_2) 
\end{array}
\right]
\endgroup .
$$
\end{lemma}

\begin{proof} 
The proof is a straightforward computation using Equation (\ref{ebchbasis}).
\end{proof}

\begin{lemma}\label{lxi012}
Let $L$ be 
a 3-dimensional $\bb{Z}_p$-Lie lattice that admits 
a basis satisfying the non-self-similarity condition, 
and choose one such basis.
Let $s_0\les s_1\les s_2$ be the $s$-invariants of $L$. 
Take $i\in \{0,1,2\}$ and $\xi \in \Xi_i$.
The following holds.
\begin{enumerate}
\item $L^\xi$ is a subalgebra of $L$ if and only if $s_i\ges 1$.
\item If $L^\xi$ is a subalgebra and, moreover, $L$ is unsolvable, 
then the $s$-invariants of $L^\xi$ are:
 \begin{enumerate}
 \item $s_0-1$, $s_1+1$, $s_2+1$, when $i = 0$.
 \item $s_0+1$, $s_1-1$, $s_2+1$, when $i = 1$.
 \item $s_0+1$, $s_1+1$, $s_2-1$, when $i = 2$.
 \end{enumerate}

\end{enumerate}
\end{lemma}

\begin{proof}
Let $\bm{x}=(x_0, x_1, x_2)$ be a basis of $L$
satisfying the non-self-similarity condition.
Part (1) follows from Lemmas \ref{lantbrmat} and \ref{lscip},
and the conditions of Definition \ref{dnonselfsc2}.
For part (2), let $\bm{y}=(y_0,y_1,y_2)$ be the basis
of $L^\xi$ given in Definition \ref{dzxiuxi},
and let $B$ be the matrix of $L^\xi$ with respect to $\bm{y}$.
Defining $z_0= [y_1,y_2]$, $z_1 = [y_2,y_0]$ and $z_2 = [y_0,y_1]$
we have  $[L^\xi, L^\xi]=\langle z_0,z_1,z_2 \rangle$ and 
$z_i = \sum_{j=0}^2 B_{ji}y_j$, $i\in\{0,1,2\}$.
Hence, the $s$-invariants of $L^\xi$ are obtained by computing the 
Smith normal form of $B$, which
can be achieved through Gaussian reduction along the rows and the columns (cf. \cite[Theorem 3.1 on page 307, Remark 3.4 on page 308]{AWalgmod}).
We consider the case of $\xi = (e)$ with $e\in \{1,...,p-1\}$,
so that $i = 0$; the others are similar and left to the reader. 
Recall that $A=\mr{diag}(a_0,a_1,a_2)$ is the notation
for the matrix of $L$ with respect to $\bm{x}$.
Observe that, since $L^\xi$
is a subalgebra  and part (1) holds, we
have $s_0\ges 1$. The matrix $B=B_{(e)}$ 
is block diagonal (Lemma \ref{lscip}),
and we see at once that $s_2+1$ is one of the 
$s$-invariants. We have to compute the Smith normal form 
of the $2\times 2$ block 
$$
\begingroup
\small
\left[
\begin{array}{r|r}
pa_0 & -ea_0 \\
-ea_0 & p^{-1}(e^2a_0+a_1)  \\
\end{array}
\right]
\endgroup
= 
\begingroup
\small
\left[
\begin{array}{rr}
u_0p^{s_0+1} & -eu_0p^{s_0} \\
-eu_0p^{s_0} & u p^{s_0-1} \\
\end{array}
\right]
\endgroup ,
$$
where $a_j = u_j p^{s_j}$ with $u_j\in\bb{Z}_p^*$
and $u = e^2u_0+u_1p^{s_1-s_0}$. It follows from the non-self-similarity
condition that $u$ is invertible in $\bb{Z}_p$. Multiplying the matrix
displayed above by 
\scalebox{0.8}{$
\left[
\begin{array}{rr}
1 & 0 \\
eu_0u^{-1}p^{} & 1 \\
\end{array}
\right]
$}
on the right, and by 
\scalebox{0.8}{$
\left[
\begin{array}{rr}
1 & eu_0u^{-1}p^{} \\
0 & 1 \\
\end{array}
\right]
$}
on the left, we get
\scalebox{0.8}{$
\left[
\begin{array}{rr}
u_0u_1u^{-1}p^{s_1+1} & 0 \\
0 & up^{s_0-1} \\
\end{array}
\right]$}.
It follows that the other two $s$-invariant are $s_0-1$ and $s_1+1$,
as desired.
\end{proof}

\vspace{5mm}

The following is the key lemma for the proof of Theorem \ref{tsl11noss}.

\begin{lemma}\label{lfrgen}
Let $L$ be a 3-dimensional unsolvable $\bb{Z}_p$-Lie lattice that
admits a basis $(x_0,x_1,x_2)$ satisfying the non-self-similarity condition. 
Let $s_0\les s_1\les s_2$ be the $s$-invariants
of $L$, $i\in \{0,1,2\}$, $\xi\in \Xi_i$, and suppose that
$M: = L^\xi$ is a subalgebra of $L$. 
Then
$$[M,M]+ p^{s_i}M = p[L,L]+p^{s_i}L.$$
\end{lemma}

\begin{proof}
Let $\bm{x}=(x_0,x_1,x_2)$, and
let $\bm{y} = (y_0,y_1,y_2)$ 
be the basis of $M$ given in Definition \ref{dzxiuxi}.
In coordinates with respect to $\bm{x}$, the submodule
$p[L,L]+p^{s_i}L$ is generated by the columns of the $3\times 6$ matrix
given in block form by $A = \left[\, pA \,|\, p^{s_i}I\,\right]$, 
where $A=\mr{diag}(a_0,a_1,a_2)$ is the matrix 
of $L$ with respect to $\bm{x}$, and 
$I$ is the $3\times 3$
identity matrix. Likewise, in coordinates with respect to $\bm{y}$,
the submodule $[M,M] + p^{s_i}M$ is
generated by the columns of the $3\times 6$ matrix
given in block form by 
$B'=\left[\, B\, |\, p^{s_i}I\,\right ]$, where $B$ is the matrix 
of $M$ with respect to $\bm{y}$.
In order to compare $p[L,L]+p^{s_i}L$ with $[M,M] + p^{s_i}M$,
we have to write generators of the latter in coordinates with respect
to $\bm{x}$. This is achieved by multiplying $B'$ 
by $U = U_\xi$ (Definition \ref{dzxiuxi}) on the left.
It follows that $p[L,L]+p^{s_i}L =[M,M] + p^{s_i}M$
if and only if $A'$ and $U B'$ have the same
Hermite normal form (reducing along the columns); 
see, for instance,
\cite[Theorem 2.9 on page 302, Theorem 2.13 on page 304]{AWalgmod}.
The Hermite normal form can be computed through 
Gaussian reduction 
\cite[Remark 3.4 on page 308]{AWalgmod}.
We treat the case $i=0$ and $\xi = (e)$ with $e\in \{1,...,p-1\}$; 
the other cases are left to the reader. 
The matrix $A'$ is easily reduced:
$$
A' = 
\begingroup
\scriptsize
\left[
\begin{array}{rrr|rrr}
pa_0 & 0 & 0 & p^{s_0} & 0 & 0\\
0 & pa_1 & 0 &0 & p^{s_0} & 0\\
0 &  0 & pa_2 & 0 & 0 & p^{s_0}
\end{array}
\right]
\endgroup
\sim
\begingroup
\scriptsize
\left[
\begin{array}{rrr|rrr}
p^{s_0} & 0 & 0 & 0 & 0 & 0\\
0 & p^{s_0} & 0 &0 &0 & 0\\
0 &  0 & p^{s_0} & 0 & 0 & 0
\end{array}
\right]
\endgroup ,
$$
where `$\sim$' denotes multiplication on the right 
by a $3\times 3$ matrix that is invertible over $\bb{Z}_p$.
Observe that this is a fancy way to say that $p[L,L]+p^{s_0}L = p^{s_0}L$
(for $i=1,2$ a similar phenomenon 
does not happen, a fact that does not change 
the difficulty of the proof). On the other hand
$$
UB' = 
\begingroup
\scriptsize
\left[
\begin{array}{rrr|rrr}
pa_0 & -ea_0 & 0 & p^{s_0} & 0 & 0\\
0 & a_1 & 0 & ep^{s_0} & p^{s_0+1} & 0\\
0 &  0 & pa_2 & 0 & 0 & p^{s_0}
\end{array}
\right]
\endgroup .
$$
The entry $p^{s_0}$ on the sixth column is used to eliminate 
the entry $pa_2$ on the third one. Moreover, the entry $p^{s_0}$
on the fourth column is used to eliminate the entries $pa_0$ and
$-ea_0$ on the first line. We get
$$ 
UB' 
\sim 
\begingroup
\scriptsize
\left[
\begin{array}{rrr|rrr}
0 & 0 & 0 & p^{s_0} & 0 & 0\\
-epa_0 & e^2a_0 + a_1 & 0 & ep^{s_0} & p^{s_0+1} & 0\\
0 &  0 & 0 & 0 & 0 & p^{s_0}
\end{array}
\right]
\endgroup .
$$
From the non-self-similarity condition, 
$v_p(e^2a_0+a_1) = v_p(a_0)=s_0$, so that the entry 
$e^2a_0+a_1$ may be used to eliminate all the other entries on the second
row. After having done that, one may permute the columns and get
the same normal form as for $A'$. 
\end{proof}

\vspace{5mm}

\noindent
\textbf{Proof of Theorem \ref{tsl11noss}.}
Let $M$ be a subalgebra of $L$ of index $p$, and let $\varphi: M\ra L$
be a virtual endomorphism of $L$. We will show  that $\varphi$ is not simple.

Suppose that  $\varphi$ is not injective. Then by  item (\ref{l3dimhji2}) of Proposition  \ref{p3dimins}, 
$\mr{ker}(\varphi)$ is 3-dimensional, so there is some $k$ such that
$I:=p^kL\subseteq \mr{ker}(\varphi)$ is a non-trivial 
$\varphi$-invariant ideal of $L$. Hence,  $\varphi$ is not simple.
Now assume that $\varphi$ is injective. Then
$\varphi$ establishes an isomorphism of $M$ with its 
image $M':=\varphi(M)$. By Proposition \ref{pisotind},  $M'$ has index $p$ in $L$ as well.
Take a  basis $\bm{x}$ of $L$
satisfying the non-self-similarity condition.
Recall the notation of Definition \ref{dzxiuxi}.
There are (unique) $\xi,\eta\in \Xi$ such that $M= L^\xi$ and $M'=L^\eta$
(Lemma \ref{lclaipsmod}). Hence,
there are unique $i,j\in\{0,1,2\}$ such that $\xi \in \Xi_i$ and $\eta \in \Xi_j$.
The rest of the proof, divided in four cases, relies on Lemma \ref{lfrgen}. 

\textit{Case 1:} $i = j$.
We take $I := p[L,L]+p^{s_i}L$, a non-trivial ideal of $L$. 
We show that $I$ is $\varphi$-invariant.
From Lemma \ref{lfrgen}, we have $[M,M]+p^{s_i}M = I = [M',M']+p^{s_i}M'$.
Observe that $I\subseteq M$.
Since $\varphi:M\rar M'$ is an isomorphism of Lie algebras, we have
$\varphi([M,M]+p^{s_i}M) = [M',M']+p^{s_i}M'$,
so that $\varphi(I) = I$.

\textit{Case 2:} $\{i, j\}=\{0,1\}$. 
We take $I: = p[L,L]+ p^{s_0}L$, a non-trivial ideal of $L$. 
We show that $I$ is $\varphi$-invariant.
First of all, we observe that $s_i = s_j$;
indeed, we can assume $i=0$ and $j=1$, so that
the minimum $s$-invariant of $M$ is $s_0-1$ (see Lemma \ref{lxi012});
if it was $s_0<s_1$ then all the $s$-invariants of $M'$
would be greater then $s_0-1$, and $M$ and $M'$
could not be isomorphic. 
From Lemma
\ref{lfrgen} we have
$[M,M]+p^{s_i}M = p[L,L] + p^{s_i}L$ and
$[M',M']+p^{s_j}M'= p[L,L]+p^{s_j}L$.
Since $s_i = s_j= s_0$, 
the same arguments as in case 1 permit
to conclude that $I\subseteq M$ and $\varphi(I) = I$.

\textit{Case 3:} $\{i, j\}=\{1,2\}$.
Taking $I: = p[L,L]+ p^{s_1}L$, the argument is  
very similar to the one of case 2. The only difference is that,
in proving that $s_i=s_j$, one has to consider the maximum value of the  
$s$-invariants instead of the minimum.

\textit{Case 4:} $\{i, j\}=\{0,2\}$.
We show that in this case we have a contradiction.
By assumption, $s_0<s_2$. We can assume $i=0$ and $j=2$. Observe that
the minimum $s$-invariant of $M$, namely $s_0-1$,
is less then all the $s$-invariants of $M'$ (Lemma \ref{lxi012}), 
so that $M\not\simeq M'$, contrary to the assumptions.
\ep


\subsection{Extension of scalars}\label{sextsca}

Throughout this section $p$ denotes an \textit{odd} prime.
Given a 3-dimensional unsolvable Lie lattice $L$
over $\bb{Z}_p$, we compute the isomorphism class of 
$L\otimes_{\bb{Z}_p} \bb{Q}_p$, a 3-dimensional unsolvable Lie algebra
over $\bb{Q}_p$ (Proposition \ref{pextsca}).
To achieve our goal, we define a $(\bb{Z}/2\bb{Z})$-valued function 
of a symmetric non-degenerate matrix $A$. This function is defined through
the classical discriminant and $\vep$-invariant of $A$, where $A$ is
thought of as a non-degenerate quadratic form
(see Section IV.2 of \cite{SerArithmetic}).
Actually, we will use the \textit{additive} $\vep$-invariant, 
which is defined through the additive version of the
Hilbert symbol, the latter being denoted by $[a,b]$ in
\cite[Remark on page 23]{SerArithmetic}.

\begin{definition}\label{detainv2}
Let $A\in gl_3(\bb{Q}_p)$ be symmetric and non-degenerate.
Let $d(A)\in \bb{Q}_p^*/(\bb{Q}_p^*)^2$  and 
$e(A)\in\bb{Z}/2\bb{Z}$ be the discriminant and the additive 
$\vep$-invariant of $A$ .
The $\eta$\textbf{-invariant} of $A$ is defined to be
$$ \eta(A): = \delta_p v_p(d(A)) + e(A)\qquad\qquad 
\eta(A)\in\bb{Z}/2\bb{Z},$$
where $\delta_p$ is the residue of $(p-1)/2$ modulo 2;
we observe that the $p$-adic valuation of 
an element of $ \bb{Q}_p^*/(\bb{Q}_p^*)^2$
is well defined in $\bb{Z}/2\bb{Z}$.
\end{definition}

\begin{lemma}\label{letainvcom}
Let $A = \mr{diag}(u_0p^{s_0}, u_1p^{s_1}, u_2p^{s_2})$
with $u_i\in\bb{Z}_p^*$ and $s_i\in\bb{Z}$ for $i\in \{0,1,2\}$.
The following holds.
\begin{enumerate}
\item $\eta(A) \equiv 1$ (mod 2) if and only if 
two of the $s$-invariants have the same parity while the
third $s$-invariant has a different one, say 
$s_i\equiv s_j\not\equiv s_k$ (mod 2), and 
$-u_iu_j$ is not a square modulo $p$.
\item For the following specific forms of $A$ (which include the canonical
forms of Remark \ref{rreci3d}), the $\eta$-invariant
is given by the following formulas, 
where $\rho \in \bb{Z}_p^*$ is not a square modulo $p$, 
$\vep_1,\vep_2 \in \{0,1\}$, $\delta_p\equiv (p-1)/2$ (mod 2),
and the congruences are modulo 2:
 \begin{enumerate}
 \item If
 $A = \mr{diag}(p^{s_0}, \rho^{\vep_1}p^{s_1},\rho^{\vep_2}p^{s_2})$ then
 $$\eta(A) \equiv \delta_p(s_0+s_1+s_2 + s_0s_1+s_0s_2+s_1s_2)+
 (\vep_1+\vep_2)s_0 + \vep_2 s_1 + \vep_1 s_2.$$
 \item If $A = \mr{diag}(p^{s_0}, -\rho^{\vep_1}p^{s_0},
 p^{s_2})$ then $\eta(A) \equiv\vep_1(s_0+s_2)$.
 \item If $A = \mr{diag}(p^{s_0}, p^{s_1},
 -\rho^{\vep_2}p^{s_1})$ then $\eta(A) \equiv \vep_2(s_0+s_1)$. 
 \item If $A = \mr{diag}(p^{s_0}, p^{s_0},p^{s_0})$ then
 $\eta(A) \equiv 0$.
 \end{enumerate}
\end{enumerate}
\end{lemma}

\begin{proof}
Given $A$, 
define $\vep_i = 0$ if $u_i$ is a square modulo $p$,
and $\vep_i = 1$ if $u_i$ is not a square modulo $p$.
A straightforward computation 
shows that  
$$\eta(A) \equiv 
\delta_p \left(\sum_{i} s_i + \sum_{i<j} s_is_j \right)
+ \sum_{i<j} (\vep_i s_j + \vep_j s_i)\qquad \mbox{(mod 2)},$$
where the indices $i,j$ take values in $\{0,1,2\}$.
(For the computation of the Hilbert symbols involved see, 
for instance, \cite[Remark on page 23]{SerArithmetic}.)
Both items of the lemma are a direct consequence 
of this formula. For item (1) one has to observe
that $-u_iu_j$ is a square modulo $p$ if and only
if $\delta_p + \vep_i +\vep_j \equiv 0$ (mod 2).
\end{proof}

\vspace{0mm}

\begin{proposition}\label{pmciffeta}
Let $A,B\in gl_3(\bb{Q}_p)$ be symmetric and non-degenerate.
Then $A$ is multiplicatively congruent to $B$ (over $\bb{Q}_p$)
if and only if $\eta(A) = \eta(B)$.
\end{proposition}

\begin{proof}
By definition, $A$ is multiplicatively congruent to $B$ 
if and only if there exists $u\in\bb{Q}_p^*$ such that $uA$ 
is congruent to $B$ (cf. Remark \ref{ralaq}). Moreover, `$uA$ congruent to $B$'
is equivalent to `$d(uA) = d(B)$ and 
$e(uA) = e(B)$'. Since $d(uA) = u d(A)$ and $e(uA) = \delta_pv_p(u)+ e(A)$, 
one can see
that `there exists $u\in\bb{Q}_p^*$ such that $uA$ 
is congruent to $B$' is equivalent to `$\eta(A) = \eta(B)$'.
\end{proof}

\begin{remark}\label{ralaq}
Analogous statements to the ones given at the 
beginning of Section \ref{scla3dill} hold over $\bb{Q}_p$.
In particular, 
the map $A\mapsto L_A$ induces a bijection between the set
of multiplicative congruence classes of
matrices of $gl_3(\bb{Q}_p)$ and the 
set
of isomorphism classes
of 3-dimensional antisymmetric algebras over $\bb{Q}_p$.
Also, the analogue of item (\ref{liliffsnzd2})
of Proposition \ref{p3dimins} holds,
so that, under $A\mapsto L_A$, symmetric non-degenerate matrices
correspond exactly to unsolvable Lie algebras.
Recall from
Example \ref{ecansl} that
there exist bases of $sl_2(\bb{Z}_p)$
and $sl_1(\Delta_p)$
such that the associated matrices 
are $A_0 = \mr{diag}(1,1,1)$
and $A_1 = \mr{diag}(1,-\rho, p)$, respectively. 
After tensoring by $\bb{Q}_p$, the same is true for 
$sl_2(\bb{Q}_p)\simeq sl_2(\bb{Z}_p)\otimes_{\bb{Z}_p}\bb{Q}_p$
and $sl_1(\bb{D}_p)\simeq sl_1(\Delta_p)\otimes_{\bb{Z}_p}\bb{Q}_p$.
One computes $\eta(A_0) = 0$ and $\eta(A_1) = 1$ (Lemma \ref{letainvcom}),
from which we see that both values of $\eta$, namely 0 and 1,
are represented.

The above observations, together with 
Proposition \ref{pmciffeta}, (re)prove that
there are exactly two isomorphism classes of 
3-dimensional unsolvable Lie algebras over $\bb{Q}_p$, 
and that $sl_2(\bb{Q}_p)$
and $sl_1(\bb{D}_p)$ represent these two classes.
\end{remark}

\vspace{0mm}

\begin{proposition}\label{pextsca}
Let $p\ges 3$ be a prime, and  
$L$ be a 3-dimensional unsolvable $\bb{Z}_p$-Lie lattice.
Let $A\in gl_3(\bb{Z}_p)\subseteq gl_3(\bb{Q}_p)$ 
be the (symmetric non-degenerate) matrix of $L$ associated with
some basis.
The following holds.
\begin{enumerate}
\item $L\otimes_{\bb{Z}_p}\bb{Q}_p \simeq sl_2(\bb{Q}_p)$ if and only if 
$\eta(A) = 0$. 
\item $L\otimes_{\bb{Z}_p}\bb{Q}_p \simeq sl_1(\bb{D}_p)$ if and only if 
$\eta(A) = 1$. 
\end{enumerate}
\end{proposition}

\begin{proof}
Observe that if the given basis of $L$ is $(x_0,x_1,x_2)$
then $(x_0\otimes 1, x_1\otimes 1, x_2\otimes 1)$ 
is a basis of $L\otimes_{\bb{Z}_p}\bb{Q}_p$
whose corresponding matrix is $A$ as well.
The proposition follows from Proposition \ref{pmciffeta} and Remark \ref{ralaq}.
\end{proof}


\subsection{Main results on Lie algebras}\label{smainla}

We collect here the main results on Lie algebras. 

\begin{theorem}\label{tsoluns}
Let $p\ges 3$ be a prime, and $L$ be 
a 3-dimensional unsolvable 
$\bb{Z}_p$-Lie lattice.
Fix $\rho\in\bb{Z}_p^*$ not a square modulo $p$.
Then $L$ is isomorphic to
exactly one of the following $\bb{Z}_p$-Lie lattices:
\begin{enumerate}
\item $L_1(s_0, s_1, s_2, \vep_1, \vep_2)$ for 
$0 \les s_0 < s_1 < s_2$ and $\vep_1,\vep_2\in\{0,1\}$, presented by:
\vspace{0mm}
$$ \langle\, x_0, x_1, x_2   \mid [x_1, x_2] =  p^{s_0}x_0, \,\, 
[x_2, x_0] =  \rho^{\vep_1}p^{s_1}x_1, \,\,
[x_0, x_1] =  \rho^{\vep_2}p^{s_2}x_2 \,\rangle .$$

\item $L_2(s_0, s_2, \vep_1)$ for $0 \les s_0  < s_2$ and 
$\vep_1 \in\{0,1\}$, presented by:
\vspace{0mm}
$$ \langle\, x_0, x_1, x_2   \mid [x_1, x_2] =  p^{s_0}x_0,  \,\,
[x_2, x_0] =  -\rho^{\vep_1}p^{s_0}x_1, \,\,
[x_0, x_1] =  p^{s_2}x_2\,\rangle .$$

\item $L_3(s_0, s_1, \vep_2)$ for $0 \les s_0  < s_1$ and 
$\vep_2 \in\{0,1\}$, presented by:
\vspace{0mm}
$$ \langle\, x_0, x_1, x_2   \mid [x_1, x_2] =  p^{s_0}x_0,  \,\,
[x_2, x_0] =  p^{s_1}x_1, \,\,
[x_0, x_1] = - \rho^{\vep_2}p^{s_1}x_2\,\rangle.$$

\item $L_4(s_0)$ for  $0 \les s_0  $, presented by: 
\vspace{0mm}
$$ \langle\, x_0, x_1, x_2   \mid [x_1, x_2] =  p^{s_0}x_0,\,\,  
[x_2, x_0] =  p^{s_0}x_1,\,\, 
[x_0, x_1] =  p^{s_0}x_2\,\rangle .$$
\end{enumerate}
Moreover, 
we have:
\begin{enumerate}
\item $L_1(s_0,s_1,s_2,\vep_1,\vep_2)$ is not self-similar of index $p$.
\item $L_2(s_0,s_2,\vep_1)$
is self-similar of index $p$ if and only if
$\vep_1 = 0$.
\item $L_3(s_0,s_1,\vep_2)$ is self-similar of index $p$ if and only if
$\vep_2 = 0$.
\item $L_4(s_0)$ is self-similar of index $p$.
\end{enumerate}
\end{theorem}

\begin{proof}
The presentations in the statement correspond exactly to the canonical forms
given in Remark \ref{rreci3d},
which indeed represent all the 3-dimensional unsolvable
$\bb{Z}_p$-Lie lattices in a non-redundant way. We proceed in two steps.

First, we show that the lattices that are stated to be 
self-similar of index $p$ admit
the required simple virtual endomorphism of index $p$. 
For this, we apply Lemma \ref{li3dssst1}.
For family (3) and $\vep_2 =0$, the lemma can be applied directly.
For family (2) and $\vep_1=0$, 
one can apply the lemma after making a change of basis
that permutes the diagonal entries; see Remark \ref{rdiagbasis}.
For family (4), one applies 
\cite[Lemma 3.4, page 115]{CasRQF}
in order to get the form $\mr{diag}(-p^{s_0}, p^{s_0},-p^{s_0})$,
so that the lemma can be applied.

Second, from Remark \ref{rcannssc} and Theorem \ref{tsl11noss},
we see that 
the lattices that are stated not 
to be self-similar of index $p$ are indeed of such a type.
This completes the proof. 
\end{proof}

\begin{theorem}\label{tssindeta0}
Let $p\ges 3$ be a prime, and $L$ be
a 3-dimensional unsolvable $\bb{Z}_p$-Lie lattice. 
If $L\otimes_{\bb{Z}_p}\bb{Q}_p\simeq 
sl_2(\bb{Q}_p)$ then $L$ is self-similar. Moreover, Table
\ref{table1} gives the value or an estimate for the 
self-similarity index $\sigma$ of $L$. 
All the 3-dimensional unsolvable $\bb{Z}_p$-Lie lattices with 
$L\otimes_{\bb{Z}_p}\bb{Q}_p\simeq sl_2(\bb{Q}_p)$ 
are represented exactly once in the table.

\setlength\extrarowheight{7pt}

\begin{table}[h!]
\begin{center}
\begin{tabular}{ |m{0.3cm}| m{3.2cm}| m{5cm} | m{4cm}| } 
\hline
&
\multicolumn{1}{c|}{{\normalfont Presentation}} 
& 
\multicolumn{1}{c|}{$\sigma$}
& 
\multicolumn{1}{c|}{{\normalfont Conditions}}
\\\hline
\multicolumn{1}{|c|}{
1.
}
&
\multicolumn{1}{c|}{
$(p^{s_0},p^{s_0},p^{s_0})$}
&
\multicolumn{1}{c|}{$\sigma = p$}
& 
\\[10pt]\hline
\multicolumn{1}{|c|}{
2.
}
&
\multicolumn{1}{c|}{
$(p^{s_0},p^{s_1},-p^{s_1})$} 
&
\multicolumn{1}{c|}{$\sigma = p$}
& 
\\[10pt]\hline
\multicolumn{1}{|c|}{
3.
}
&
\multicolumn{1}{c|}{
$(p^{s_0},p^{s_1},-\rho p^{s_1})$}
&
\multicolumn{1}{c|}{
$p^2\les \sigma \les p^{s_1-s_0 + 1}$
}
& 
\multicolumn{1}{c|}{
$s_0\equiv_2 s_1 $
}
\\[10pt]\hline
\multicolumn{1}{|c|}{
4.
}
&
\multicolumn{1}{c|}{
$(p^{s_0},-p^{s_0},p^{s_2})$ }
&
\multicolumn{1}{c|}{
$\sigma = p$
}
& 
\\[10pt]\hline
\multicolumn{1}{|c|}{
5.
}
&
\multicolumn{1}{c|}{
$(p^{s_0},-\rho p^{s_0},p^{s_2})$ }
&
\multicolumn{1}{c|}{
$p^2\les \sigma \les p^{\frac{s_2-s_0}{2} + 1}$
}
& 
\multicolumn{1}{c|}{
$s_0\equiv_2 s_2 $
}
\\[10pt]\hline
\multicolumn{1}{|c|}{
6.
}
&
\multicolumn{1}{c|}{
$(p^{s_0},\rho^{\vep_1}p^{s_1},\rho^{\vep_2}p^{s_2})$}
&
\multicolumn{1}{c|}{
$p^2\les \sigma \les p^{\frac{s_1-s_0}{2}+\frac{s_2-s_0}{2} + 1}$
}
&
\multicolumn{1}{c|}{
$s_0\equiv_2 s_1\equiv_2 s_2$
}
\\[10pt]\hline
\multicolumn{1}{|c|}{
7.
}
&
\multicolumn{1}{c|}{
$(p^{s_0},\rho^{\vep_1}p^{s_1},\rho^{\vep_2}p^{s_2})$}
&
\multicolumn{1}{c|}{
$p^2\les \sigma \les p^{\frac{s_1-s_0}{2} + 1}$
}
&
\multicolumn{1}{c|}{
$
\begin{array}{l}
s_0\equiv_2 s_1\not \equiv_2 s_2\\[-5pt]
\vep_1+\delta_p \equiv_2 0
\end{array}
$
}
\\[20pt]\hline
\multicolumn{1}{|c|}{
8.
}
&
\multicolumn{1}{c|}{
$(p^{s_0},\rho^{\vep_1}p^{s_1},\rho^{\vep_2}p^{s_2})$}
&
\multicolumn{1}{c|}{
$p^2\les \sigma \les p^{\frac{s_2-s_1}{2} + 1}$
}
&
\multicolumn{1}{c|}{
$
\begin{array}{l}
s_1\equiv_2 s_2\not \equiv_2 s_0\\[-5pt]
\vep_1+\vep_2+\delta_p \equiv_2 0
\end{array}
$
}
\\[20pt]\hline
\multicolumn{1}{|c|}{
9.
}
&
\multicolumn{1}{c|}{
$(p^{s_0},\rho^{\vep_1}p^{s_1},\rho^{\vep_2}p^{s_2})$}
&
\multicolumn{1}{c|}{
$p^2\les \sigma \les p^{\frac{s_2-s_0}{2} + 1}$
}
&
\multicolumn{1}{c|}{
$
\begin{array}{l}
s_0\equiv_2 s_2\not \equiv_2 s_1\\[-5pt]
\vep_2+\delta_p \equiv_2 0
\end{array}
$
}
\\[20pt]\hline
\end{tabular}
\end{center}
\caption{
\textit{Estimates for the self-similarity index
as in Theorem \ref{tssindeta0}.
In the table, $\rho \in \bb{Z}_p^*$ is a fixed non-square modulo $p$,
and $\delta_p \equiv_2 (p-1)/2$,
where the symbol $\equiv_2$ denotes congruence modulo 2.
The variables appearing in each line of the table 
(i.e., the parameters of the classification) take values:
$s_0,s_1,s_2\in\bb{N}$ with $s_0<s_1<s_2$ and $\vep_1,\vep_2\in\{0,1\}$.
The first column of the table gives the diagonal entries of 
the (diagonal) matrix associated with the canonical presentation
of $L$ as in Theorem \ref{tsoluns}. 
The second column gives the value of $\sigma$ or an estimate for it.
The third column gives conditions on the parameters in order the
given estimate to be true.
}
}

\label{table1}
\end{table}
\end{theorem}

\begin{proof}
Observe that one may use Lemma \ref{letainvcom}
in order to compute the $\eta$-invariant of the matrices corresponding
to the canonical presentations (appearing, for 
instance, in Theorem \ref{tsoluns})
of all the 3-dimensional unsolvable $\bb{Z}_p$-Lie lattices. 
One can check that the canonical
presentations appearing in the table are exactly the ones
for which $\eta = 0$. It follows (Proposition \ref{pextsca}) 
that all the 3-dimensional unsolvable Lie lattices
with $L\otimes_{\bb{Z}_p} \bb{Q}_p\simeq sl_2(\bb{Q}_p)$ are represented exactly once.

Observe that lines 1, 2 and 4 of the table correspond 
exactly to the index-$p$ self-similar lattices respectively of cases
4, 3 and 2 of Theorem \ref{tsoluns}.
These are exactly the cases where $L$
is self-similar of index $p$, so that the given value of $\sigma$ is correct,
and the estimate
$p^2\les \sigma$ of the other lines of the table is correct as well.

It remains to prove the validity of the given upper bounds for
$\sigma$.
Here the idea is to exhibit a
3-dimensional subalgebra $M$ of $L$ that is self-similar of
index $p$. Applying Lemma \ref{lsubss}, we get that $L$
is self-similar of index $p[L:M]$.
We proceed as follows.
Let $(x_0,x_1,x_2)$ be a basis of $L$ that puts $L$ in its canonical form.
We denote by $A=\mr{diag}(a_0,a_1,a_2)$ the corresponding matrix.
In each case, we define $M = \langle p^{k_0}x_0,p^{k_1}x_1, p^{k_2}x_2 \rangle$,
where the exponents $k_0,k_1,k_2\in\bb{N}$ are given below case by case. 
Observe that, for any choice of these
exponents, the matrix $B$ of formula (\ref{ebchbasis}) of Lemma \ref{lantbrmat}
(where $U=\mr{diag}(p^{k_0},p^{k_1}, p^{k_2})$)
is 
$$B= \mr{diag}(a_0p^{-k_0+k_1+k_2},a_1p^{k_0-k_1+k_2},a_2p^{k_0+k_1-k_2}).$$
For line 9 of the table we make the full proof
while, for the other relevant lines, we just give the exponents, leaving 
the details to the reader.

\vspace{5pt}
\noindent
\textit{Line 3.} Take $(k_0,k_1,k_2)=
\left(0,\frac{s_1-s_0}{2},\frac{s_1-s_0}{2}\right)$.

\vspace{5pt}

\noindent
\textit{Line 5.} Take $(k_0,k_1,k_2)=
\left(0,0,\frac{s_2-s_0}{2}\right)$.

\vspace{5pt}

\noindent
\textit{Line 6.} Take $(k_0,k_1,k_2)=
\left(0,\frac{s_1-s_0}{2},\frac{s_2-s_0}{2}\right)$.

\vspace{5pt}

\noindent
\textit{Line 7.} Take $(k_0,k_1,k_2)=
\left(0,\frac{s_1-s_0}{2},0\right)$.

\vspace{5pt}

\noindent
\textit{Line 8.} Take $(k_0,k_1,k_2)=
\left(0,0, \frac{s_2-s_1}{2}\right)$.

\vspace{5pt}

\noindent
\textit{Line 9.} Take $(k_0,k_1,k_2)=
\left(0,0, \frac{s_2-s_0}{2}\right)$.
Define $l:=\frac{s_2-s_0}{2}$ and observe
that $l\ges 1$ is an integer.
For the corresponding $M$ (a priori just a submodule of $L$)
we have $[L:M]=p^l$. We have to show that $M$ is a subalgebra
of $L$ and that $M$ is self-similar of index $p$.
One computes
$B = \mr{diag}(p^{s_0 + l}, \rho^{\vep_1}p^{s_1+l},\rho^{\vep_2}p^{s_0 + l})$.
The entries of $B$ are in $\bb{Z}_p$, so that $M$
is a subalgebra of $L$. We will show that 
the canonical matrix of $L$ is 
$\mr{diag}(p^{s_0+l}, -p^{s_0+l}, p^{s_1+l})$,
which implies that $M$ is self-similar of index $p$ (Theorem \ref{tsoluns}).
Indeed, we can permute the diagonal 
entries of $B$, multiply them by $\rho^{\vep_1}$
(see Remark \ref{rdiagbasis}) and ``eliminate'' a square (through a congruence), 
getting $\mr{diag}(\rho^{\vep_1}p^{s_0+l}, \rho^{\vep_1+\vep_2}p^{s_0+l},
p^{s_1+l})$ as the matrix of $M$ with respect to some basis.
With other changes of basis we can first ``move'' the factor $\rho^{\vep_1}$
of the first diagonal entry to the second (cf. \cite[Lemma 3.4, page 115]{CasRQF}), 
and then eliminate the resulting square $\rho^{2\vep_1}$.
By writing $-1 =\rho^{\delta_p}u^2$ with $u\in\bb{Z}_p^*$ and 
eliminating $u^2$, we get
$\mr{diag}(p^{s_0+l}, - \rho^{\vep_2+\delta_p}p^{s_0+l},
p^{s_1+l})$. Since $\vep_2+\delta_p\equiv_2 0$,
possibly eliminating another square, 
we arrive at the desired presentation, 
and the proof of this case is complete.
\end{proof}

\begin{theorem} \label{tsl2}
Let $p\ges 3$ be a prime.
The following holds.
\begin{enumerate}
\item The $\bb{Z}_p$-Lie lattice $sl_2(\bb{Z}_p)$ and its congruence subalgebras
$sl_2^k(\bb{Z}_p)= p^ksl_2(\bb{Z}_p)$, $k\ges 1$, are self-similar of index $p$. 
\item 
\begin{enumerate} 
\item The $\bb{Z}_p$-Lie lattice $sl_2^{\sylow}(\bb{Z}_p)$ and the terms 
$\gamma_k(sl_2^{\sylow}(\bb{Z}_p))$, $k\ges 1$, of its
lower central series are self-similar of index $p$.
\item If $I\subseteq L$ is a non-zero ideal of $sl_2^{\sylow}(\bb{Z}_p)$
then $I$ has dimension 3 and it is self-similar of index $p$ or $p^2$.
\end{enumerate}
\end{enumerate}
\end{theorem}

\begin{proof} First,
the canonical forms of $sl_2(\bb{Z}_p)$, $sl_2^k(\bb{Z}_p)$, 
$sl_2^{\sylow}(\bb{Z}_p)$ and $\gamma_k(sl_2^{\sylow}(\bb{Z}_p))$
have been computed in Example \ref{ecansl}. 
According to Theorem \ref{tsoluns}, they all
correspond to lattices that are self-similar of index $p$.

Second, since $L$ is a 3-dimensional unsolvable Lie lattice
then any non-zero ideal $I$ of $L$ is 3-dimensional (item (\ref{l3dimhji2}) 
of Proposition \ref{p3dimins}).
If $I = L$ then $I$ is self-similar of index $p$. Assume $I\neq L$.
Since $\gamma_{2m}(L) = p^mL$ for all $m\ges 0$ (Lemma \ref{lgenGnL}),
there exists $k\in \bb{N}$ 
such that $\gamma_{k}(L)\subseteq I$ 
(where $I\neq L$ implies that $k\ges 1$). Taking the least such $k$,
we have $\gamma_{k-1}(L)\not\subseteq I$, so that 
by Proposition \ref{pinsresnil2} we have $I \subseteq \gamma_{k-1}(L)$.
By minimality of $k$, we end up with $\gamma_k(L)\subseteq I\subset \gamma_{k-1}(L)$,
where $\subset$ denotes strict inclusion.
Observe that $[\gamma_{k-1}(L):\gamma_{k}(L)]$ is $p$
if $k$ is even and it is $p^2$ if $k$ is odd.
In any case $[I:\gamma_{k}(L)]$ is $1$ or $p$.
Since $\gamma_k(L)$ is self-similar of index $p$ then
$I$ is self-similar of index $p$ or $p^2$ (Lemma \ref{lsubss}).
\end{proof}

\begin{theorem}\label{tsl1tnss}
Let $p\ges 3$ be a prime, and $L$ be a 3-dimensional unsolvable 
$\bb{Z}_p$-Lie lattice. 
Suppose that $L\otimes_{\bb{Z}_p}\bb{Q}_p \simeq sl_1(\bb{D}_p)$.
Then $L$ is not self-similar of index $p$.
\end{theorem}

\begin{proof}
Let $A$ be the matrix of $L$ associated with the canonical presentation
given in Theorem \ref{tsoluns}.
From Proposition \ref{pextsca}, 
we see that $\eta(A)=1$. From Lemma \ref{letainvcom}, 
the value of $\eta(A)$ can be explicitly computed.
We have to show that the canonical presentation of $L$ belongs
to one of the families that is proven to be non self-similar of index $p$
in Theorem \ref{tsoluns}.
If $s_0=s_1=s_2$ then $\eta(A)=0$, a contradiction.
If $s_0=s_1<s_2$ then $\eta(A)=1$ implies $\vep_1(s_0+s_2)\equiv 1$ (mod 2), 
so that $\vep_1 = 1$, as desired.
If $s_0<s_1 = s_2$ then 
$\vep_2(s_0+s_1)\equiv 1$ (mod 2), 
so that $\vep_2 = 1$, as desired.
Finally, if $s_0<s_1<s_2$ then there is nothing to prove, since no such $L$ is
self-similar of index $p$.
\end{proof}

\vspace{5mm}
We expect a stronger version of Theorem \ref{tsl1tnss} to be true.

\begin{conjecture}\label{conjsl1alg}
Let $p\ges 3$ be a prime, and $L$ be a 3-dimensional unsolvable $\bb{Z}_p$-Lie lattice. 
If $L\otimes_{\bb{Z}_p}\bb{Q}_p \simeq sl_1(\bb{D}_p)$
then $L$ is not self-similar.
\end{conjecture}


\section{Proofs of the main theorems}\label{s3dimins}
 We quickly recall Lazard's correspondence, which is the main tool in this section. With a saturable pro-$p$ group $G$ one may associate a saturable $\mathbb{Z}_p$-Lie lattice $L_G$ in the following way: $G$ and $L_G$    are identified as sets, and the Lie operations are defined by
\begin{displaymath}
g+h=\lim_{n \to \infty}(g^{p^n}h^{p^n})^{p^{-n}},  ~ ~ ~  [g,h]_{\mr{Lie}}=\lim_{n \to \infty}[g^{p^n},h^{p^n}]^{p^{-2n}}=\lim_{n\to \infty} (g^{-p^n}h^{-p^n}g^{p^n}h^{p^n})^{p^{-2n}} .
\end{displaymath}
On the other hand, if $L$ is a saturable $\mathbb{Z}_p$-Lie lattice, then the Campbell-Hausdorff formula induces a group structure on $L$; the resulting group is a saturable pro-$p$ group. If this construction is applied to the $\mathbb{Z}_p$-Lie Lattice $L_G$ associated with a saturable group $G$, one recovers the original group. Indeed, the assignment $G\mapsto L_G$ gives an isomorphism between the category of saturable pro-$p$ groups and the category of saturable $\mathbb{Z}_p$-Lie lattices. See \cite[IV (3.2.6)]{Laz65}, \cite[Section 2]{KloLie05} and 
\cite{GSKpsdimJGT} for more details.

\begin{theorem}\label{tGSK2}
Let $p$ be a prime and $G$ be a saturable $p$-adic analytic pro-$p$ group 
of dimension $\mr{dim}(G)\les p$, and $L_G$ be the
$\bb{Z}_p$-Lie lattice associated with $G$. Let $D,N\subseteq G$ be subsets, and denote them by $L_D$, $L_N$ when regarded as subsets of $L_G$.
Then the following holds.
\begin{enumerate}
\item If $D$ is a closed subgroup of $G$ then $D$ is saturable.
\item $D$ is an open subgroup of $G$ if and only if 
$L_D$ is a finite-index subalgebra  of $L_G$. If this is the case, then  $[G:D]=[L_G: L_D]$.
\item The set $N$ is a closed normal subgroup of $G$ if and only if $L_N$ is an ideal  
of $L_G$. 
\item If $D$ is an open subgroup of $G$ and $\varphi: D\rar G$ is a function, then:
 \begin{enumerate}
 \item $\varphi$ is a group homomorphism if and only if $\varphi$ is
 a Lie algebra homomorphism (denoted by $L_\varphi$).
 \item If $\varphi$ is a group homomorphism, then
 $\varphi$ is simple if and only if $L_\varphi$ is simple.
 \end{enumerate}
\end{enumerate}
\end{theorem}

\begin{proof}
Item (1) follows from  \cite[Theorem A]{GSKpsdimJGT} and  \cite[Proposition D]{GSKpsdimJGT}. Let $D$ be an open subgroup of $G$. By (1), $D$ is saturable and $L_D$ is a saturable subalgebra of $L_G$. Since $G$ is a saturable pro-$p$ group, it admits a valuation 
$\omega : G \rightarrow \mathbb{R}_{>0} \cup\{\infty \}$ such that 
$(G, \omega)$ is a saturated pro-$p$ group (cf. \cite[Section 2]{KloLie05}). 
In particular, the sets 
$G_{\mu} = \{ g\in G ~ | ~ \omega(g) \ges \mu \}$, $\mu \in \mathbb{R}_{>0}$, 
are open normal subgroups of $G$ and form a basis
of neighborhoods of the identity. 
It follows that there exists $\nu \in\bb{R}_{>0}$ such that
$G_\nu$ is contained in $D$, so that we have a chain
$G_\nu\subseteq D\subseteq G$ 
of saturable groups.
From Lazard's correspondence
there is an associated
chain $L_{G_\nu}\subseteq L_D\subseteq L_G$ of Lie algebras.
Now, from \cite[Proposition A.2]{KloLie05}, for every $y\in G$ 
(in particular, for every $y\in D$) the multiplicative coset $yG_{\nu}$ 
is equal to the additive coset $y + L_{G_{\nu}}$. Hence, 
$[G : G_{\nu}] = [L_G : L_{G_{\nu}}]$   and  
$[D : G_{\nu}] = [L_D : L_{G_{\nu}}]$, which yields
 $[G : D] = [L_G : L_D]$. Moreover, from this fact and   \cite[Theorem E]{GSKpsdimJGT}  it follows that $D$ is an open subgroup of $G$ whenever $L_D$ is a finite index subalgebra of $L_G$. This establishes (2). Note that (3) follows from \cite[Theorem E]{GSKpsdimJGT}. Part (a) of (4) follows from (1) and Lazard's correspondence. Finally, to establish part (b) of (4) it suffices to observe  that
the $\varphi$-invariance of a subset $N\subseteq G$
is a purely set theoretical notion, namely, it does not depend on
whether we look at $N$ in the group context or in the Lie-algebra context.
\end{proof}

\vspace{5mm}

\noindent
\textbf{Proof of Proposition \ref{pgsstlss2}.}
The proposition is a direct consequence of
Proposition \ref{pssiffve} and Theorem \ref{tGSK2}.\ep

\vspace{5mm}

\noindent
\textbf{Proof of Theorem \ref{tmain3}.}
The Lie lattices presented in the statement form
a subset of the ones 
given in Theorem \ref{tsoluns}   
(where all the isomorphism classes of
3-dimensional unsolvable $\bb{Z}_p$-Lie lattices are represented in a non-redundant way). 
More precisely,
it is the subset of the residually nilpotent lattices  
(cf. Lemma \ref{lgenGnL}). 
Since $p$ is greater than the dimension of the groups and algebras involved,
it follows from \cite[Theorem B]{GSKpsdimJGT} that the groups given in the statement
constitute the claimed complete and irredundant list.
The conclusions about self-similarity of index $p$ follow
directly from Proposition \ref{pgsstlss2} and Theorem \ref{tsoluns}.\ep

\vspace{5mm}

\noindent
\textbf{Proof of Theorem \ref{tmainsl2}.}
For $p\ges 5$, respectively, $p\ges 3$, 
the groups $SL_2^{\sylow}(\bb{Z}_p)$, respectively, $SL_2^1(\bb{Z}_p)$
are 3-dimensional unsolvable saturable 
$p$-adic analytic pro-$p$ groups.
The theorem follows from Proposition \ref{pgsstlss2}
and Theorems \ref{tssindeta0} and \ref{tsl2} (see also Theorem \ref{tGSK2} and,
for the $\bb{Z}_p$-Lie algebras associated with the groups involved, 
\cite[Section 7.3]{GSKpsdimJGT}, \cite[page 158]{IlaZfunSL2} and 
\cite[Theorem B, item (3)]{GSpsat}).\ep

\vspace{5mm}

\noindent
\textbf{Proof of Theorem \ref{tmainsl1}.}
For $p\ges 5$, respectively, $p\ges 3$, 
the groups $SL_1^1(\Delta_p)$, respectively, $SL_1^2(\Delta_p)$
are 3-dimensional unsolvable saturable
$p$-adic analytic pro-$p$ groups.
The theorem follows from Proposition \ref{pgsstlss2}
and Theorem \ref{tsl1tnss} 
(for the $\bb{Z}_p$-Lie algebras associated with the groups involved,
see \cite[Section 7.3]{GSKpsdimJGT} and \cite[Section 2]{KloSL1}).
\ep

\begin{corollary}\label{ccomppadic}
Let $p\ges 3$ be a prime and let $G$ be a compact $p$-adic analytic group
whose associated $\bb{Q}_p$-Lie algebra $\cl{L}_G$ is isomorphic to 
$sl_2(\bb{Q}_p)$. Then $G$ is self-similar.
\end{corollary}

\begin{proof}
Since $G$ is  a compact $p$-adic analytic group, there exists an open subgroup
$H$ of $G$ that is a saturable pro-$p$ group; moreover, $L_H\otimes_{\bb{Z}_p}\bb{Q}_p \simeq  \cl{L}_G \simeq sl_2(\bb{Q}_p)$ (cf. \cite[Section 9.5]{DixAnaProP}). Now the  corollary follows from
Theorem \ref{tssindeta0}, Proposition \ref{pgsstlss2} and 
Corollary \ref{cHssthenGss}.
\end{proof}

\begin{remark}\label{rleveltrans}
Let $p\ges 3$ be a prime and let $G$ be 
a 3-dimensional unsolvable
torsion-free $p$-adic analytic pro-$p$ group that is self-similar of index $p$.
Take a faithful self-similar action of $G$ on $T_p$ that is transitive on the 
first level. We claim that this action is level transitive. Indeed, there exists a basis $(x_0,x_1,x_2)$ of $L_G$ such that $L_D= \langle x_0, px_1, x_2\rangle$ is a subalgebra of $L_G$ of index $p$ and the homomorphism $\varphi :L_D \rar L_G$ given by 
$\varphi(x_0) = x_0$, $\varphi(px_1)= x_1$ and 
$\varphi(x_2) =px_2$ is a simple virtual endomorphism of index
$p$ (cf. Lemma \ref{li3dssst}). By Theorem \ref{tGSK2}, $\varphi: D\rar G$ is a simple
virtual endomorphism of $G$ of index $p$. It is easy to see that $\varphi(D_{n+1})\not\subseteq D_{n+1}$ for all $n\ges 0$, where $D_n$ is the domain of the $n$-th power of $\varphi$. By Lemma \ref{lphiregular},  $\varphi$ is regular,
so the action of $G$ on $T_p$ is level transitive 
(cf. \cite[Proposition 4.20]{NekVirEnd}).
\end{remark}


\begin{footnotesize}

\providecommand{\bysame}{\leavevmode\hbox to3em{\hrulefill}\thinspace}
\providecommand{\MR}{\relax\ifhmode\unskip\space\fi MR }
\providecommand{\MRhref}[2]{%
  \href{http://www.ams.org/mathscinet-getitem?mr=#1}{#2}
}
\providecommand{\href}[2]{#2}

\end{footnotesize}

\end{document}